\documentclass{article}
\usepackage{amsmath,amssymb,amsthm}
\usepackage[all]{xy}
\usepackage[usenames]{color}
\usepackage{hyperref}

\numberwithin{equation}{section}
\newtheorem{thm}{Theorem}[section]
\newtheorem{cor}[thm]{Corollary}
\newtheorem{prop}[thm]{Proposition}
\newtheorem{lem}[thm]{Lemma}
\theoremstyle{definition}
\newtheorem{defn}[thm]{Definition}
\newtheorem{question}[thm]{Question}
\theoremstyle{remark}
\newtheorem{rmk}[thm]{Remark}

\newcommand{\co}{\colon\thinspace}
\newcommand{\mb}[1]{\mathbb{#1}}

\newcommand{\Ext}{\ensuremath{{\rm Ext}}}

\newcommand{\Hom}{\ensuremath{{\rm Hom}}}

\newcommand{\colim}{\ensuremath{\mathop{\rm colim}}}

\newcommand{\holim}{\ensuremath{\mathop{\rm holim}}}

\newcommand{\overto}{\mathop\rightarrow}

\newcommand{\longoverto}{\mathop{\longrightarrow}}

\newcommand{\Map}{\ensuremath{{\rm Map}}}
\newcommand{\GL}{{\rm GL}}

\newcommand{\PSL}{{\rm PSL}}

\newcommand{\Aut}{{\rm Aut}}

\newcommand{\Spec}{{\rm Spec}}

\newcommand{\Spf}{{\rm Spf}}
\newcommand{\tmf}{{\rm tmf}}

\newcommand{\TMF}{{\rm TMF}}
\newcommand{\Tmf}{{\rm Tmf}}

\newcommand{\psth}{\psi\mbox{-}\theta}
\newcommand{\alg}{\mbox{-}alg}

\newcommand{\mell}{{\cal M}_{ell}}
\newcommand{\Mell}{\overline{\cal M}_{ell}}
\newcommand{\Mlog}{\overline{\cal M}_{log}}
\newcommand{\Wlog}{{\cal W}}
\newcommand{\MMM}{\overline{\cal M}}
\newcommand{\multmod}{{\cal M}_{\mb G_m}}

\newcommand{\mtate}{{\cal M}_{Tate}}
\newcommand{\mfg}{{\cal M}_{fg}}
\newcommand{\LSc}{{\bf LogSch}}
\newcommand{\ALS}{{\bf AffLog}}
\newcommand{\Sch}{{\bf Sch}}
\newcommand{\Et}{{\acute{E}t}}

\newcommand{\dlog}{{\rm dlog}}

\newcommand{\smsh}[1]{\ensuremath{\mathop{\wedge}_{#1}}}

\newcommand{\comp}[1]{\ensuremath{#1^\wedge}}

\newcommand{\pow}[1]{[\![{#1}]\!]}

\title{Topological modular forms with level structure}
\author{
  Michael Hill\thanks{Partially supported by NSF DMS--0906285, DARPA
    FA9550--07--1--0555, and the Sloan foundation.},
  Tyler Lawson\thanks{Partially supported by NSF DMS--1206008 and the
    Sloan foundation.}
}
\begin{document}
\maketitle
\begin{abstract}
  The cohomology theory known as $\Tmf$, for ``topological modular
  forms,'' is a universal object mapping out to elliptic cohomology
  theories, and its coefficient ring is closely connected to the
  classical ring of modular forms.  We extend this to a functorial
  family of objects corresponding to elliptic curves with level
  structure and modular forms on them.  Along the way, we produce a
  natural way to restrict to the cusps, providing multiplicative maps
  from $\Tmf$ with level structure to forms of $K$-theory.  In
  particular, this allows us to construct a connective spectrum
  $\tmf_0(3)$ consistent with properties suggested by Mahowald and
  Rezk.

  This is accomplished using the machinery of logarithmic structures.
  We construct a presheaf of locally even-periodic elliptic cohomology
  theories, equipped with highly structured multiplication, on the
  log-\'etale site of the moduli of elliptic curves.  Evaluating this
  presheaf on modular curves produces $\Tmf$ with level structure.
\end{abstract}


\section{Introduction}

The subject of topological modular forms traces its origin back to the
Witten genus.  The Witten genus is a function taking String manifolds
and producing elements of the power series ring $\mb C\pow{q}$, in a
manner preserving multiplication and bordism classes (making it a
genus of String manifolds).  It can therefore be described in terms of
a ring homomorphism from the bordism ring $MO\langle 8\rangle_*$ to
this ring of power series.  Moreover, Witten explained that this
should factor through a map taking values in a particular subring
$MF_*$: the ring of modular forms.

An algebraic perspective on modular forms is that they are universal
functions on elliptic curves.  Given a ring $R$, an elliptic curve
${\cal E}$ over $R$, and a nonvanishing invariant $1$-form $\omega$ on
${\cal E}$, a modular form $g$ assigns an invariant $g({\cal
  E},\omega) \in R$.  This is required to respect base change for maps
$R \to R'$ and be invariant under isomorphisms of elliptic curves
${\cal E} \to {\cal E}'$.  The form is of weight $k$ if it satisfies
$g({\cal E},\lambda \omega) = \lambda^{-k} g({\cal
  E},\omega)$. Modular forms can be added and multiplied, making them
into a graded ring.

Thus, given a graded ring $E_*$, we might have two pieces of data.

First, this ring may arise as the coefficient ring of a complex
orientable cohomology theory $E$.  This gives $E$ a theory of Chern
classes, and from this we construct a formal group $\mb G_E$ on $E_*$
\cite{quillen-fgl}.

Second, the ring $E_*$ may carry an elliptic curve ${\cal E}$ with an
invariant $1$-form $\omega$ in a manner appropriately compatible with
the grading.  This elliptic curve also produces a formal group
$\widehat{\cal E}$ on $E_*$.

The cohomology theory becomes an {\em elliptic cohomology theory} if
we also specify an isomorphism $\mb G_E \to \widehat{\cal E}$
\cite{ando-hopkins-strickland-witten}.  The defining properties of
modular forms automatically produce a map of graded rings $MF_* \to
E_*$.  Witten's work then suggests a topological lift: elliptic
cohomology theories should possess a map $MO\langle 8\rangle_* \to
E_*$.  We can ask if this comes from a natural map $MO\langle 8\rangle
\to E$, and further if the factorization $MO\langle 8\rangle_* \to
MF_* \to E_*$ has a natural topological lift.  The subjects of
elliptic cohomology and topological modular forms were spurred by the
these developments \cite{lnm-elliptictopology,
  landweber-ravenel-stong, hopkins-mahowald-elliptic}.

However, within this general framework there are several related
objects that have been interchangeably described by these names.

The original definitions of elliptic cohomology (${\cal E}ll$) and of
topological modular forms ($\tmf$) produced ring spectra, generating
multiplicative cohomology theories, that enjoy universal properties
for maps to certain elliptic cohomology theories.

Precisely describing this universality leads to a more powerful
description as a limit.  One would like to take a ring $E_*$ equipped
with an elliptic curve ${\cal E}$ and solve a realization problem,
functorially, to produce a spectrum representing a corresponding 
elliptic cohomology theory \cite{goerss-landweber-families}.  These
representing objects are called elliptic spectra.  After producing a
sufficiently large diagram of such elliptic spectra, the homotopy
limit is a ring spectrum of topological modular forms.

Even further, these functorial elliptic spectra satisfy a patching
property: many limit diagrams of rings equipped with elliptic curves
are translated into homotopy limit diagrams of elliptic spectra.  This
grants one the power to extend to a larger functor: a presheaf of
spectra on a moduli stack of elliptic curves. The homotopy limit
property is encoded by fibrancy in a Jardine model structure.

There are several versions of this story.  The division depends on
whether one allows only elliptic spectra corresponding to smooth
elliptic curves (represented by the Deligne-Mumford stack $\mell$),
elliptic curves with possible nodal singularities (represented by the
compactification $\Mell$), or general cubic curves (represented by an
algebraic stack ${\cal M}_{cub}$).  These are sometimes informally
given the names $\TMF$, $\Tmf$, and $\tmf$ respectively, and they
represent a progressive decrease in our ability to obtain conceptual
interpretations or construct objects.  The stack $\mell$ can be
extended to a derived stack representing derived elliptic curves
\cite{lurie-survey}; the so-called ``old'' construction of topological
modular forms due to Goerss-Hopkins-Miller, by obstruction theory,
gives the \'etale site of $\Mell$ a presheaf of elliptic spectra
\cite{mark-construct}; the less-conceptual process of taking a
connective cover of $\Tmf$ produces a single spectrum $\tmf$.  The
generalizations of $\tmf$ that exist have exceptionally interesting
properties, but are constructed in a somewhat ad-hoc manner (however,
see \cite{mathew-tmfhomology}).

There are many situations where extra functoriality for elliptic
spectra can be a great advantage \cite{mark-buildings,
  stojanoska-duality, meier-unitedelliptic}.  In particular,
considering elliptic curves equipped with extra structure, such as
choices of subgroups or torsion points, leads to a family of
generalizations of $\mell$.  The corresponding modular curves ${\cal
  M}(n)$, ${\cal M}_0(n)$, ${\cal M}_1(n)$, and more are extensively
studied from the points of view of number theory and arithmetic
geometry.  Away from primes dividing $n$, these automatically inherit
presheaves of elliptic spectra from $\TMF$.  For example, the maps ${\cal
  M}(n)[1/n] \to \mell[1/n]$ are \'etale covers with Galois group
$\GL_2(\mb Z/n)$, and so the \'etale presheaf defined on $\mell$ can be
evaluated on ${\cal M}(n)[1/n]$ to produce a spectrum $\TMF(n)$ with a
$\GL_2(\mb Z/n)$-action, and the homotopy fixed-point spectrum is
$\TMF[1/n]$.

However, the compactifications $\MMM(n)$ are not \'etale over $\Mell$.
In complex-analytic terms, if $\Mell$ is described in terms of a
coordinate $\tau$ on the upper half-plane, then the compactification
point or ``cusp'' has a coordinate $q = e^{2\pi i \tau}$; the cusps in
$\MMM(n)$ have coordinates expressed in terms of $q' = e^{2\pi i
  \tau/n}$, satisfying $q = (q')^n$.  Since the coefficient
$(q')^{n-1}$ in the expression $dq = n (q')^{n-1} dq'$ is not a unit,
the map is not an isomorphism on cotangent spaces over the cusp of
$\Mell$.  This ramification complicates the key input to the
Goerss-Hopkins-Miller obstruction theory.

The obstruction theory constructing $\Tmf$ does generalize to each
individual moduli stack, provided that if one imposes level structure
at $n$ one must first invert $n$.  This has been carried out in some
instances to construct objects $\Tmf(n)$ \cite{stojanoska-duality}.
However, needing to reconstruct $\Tmf$ once per level structure is
less than satisfying, and does not provide any functoriality across
different forms of level structure.  Moreover, it does not give an
immediate reason why one might expect a relationship between $\Tmf$
and the homotopy fixed-point object for the action of $\GL_2(\mb Z/n)$
on $\Tmf(n)$.

However, these ramified maps $\MMM(n) \to \Mell$ do possess a slightly
less restrictive form of regularity.  The cusp determines a
``logarithmic structure'' on $\Mell$ in the sense of Kato
(\S\ref{sec:loggeom}), and the various maps between moduli of
elliptic curves are log-\'etale---this roughly expresses the fact that
in the expression $\dlog(q^n) = n\cdot\dlog(q)$, the coefficient is a
unit away from primes dividing $n$.  These ramified covers form part
of a log-\'etale site enlarging the \'etale site of $\Mell$.
Moreover, for the log schemes we will be considering, the fiber
product in an appropriate category of objects with logarithmic
structure is geared so that the \v Cech nerve of the cover $\Mell(n)
\to \Mell$ is the simplicial bar construction for the action of
$\GL_2(\mb Z/n)$ on $\Mell(n)$ (see \ref{prop:cechnerve}).

These precise properties turn out to be useful from the point of view
of extending functoriality for topological modular forms.  In this
paper, we use them to establish the existence of topological modular
forms spectra for all modular curves (Theorem~\ref{thm:main}), as
follows.

For a fixed integer $N$ and a subgroup $\Gamma < \GL_2(\mb Z/N)$ we
will construct an $E_\infty$ ring spectrum $\Tmf(\Gamma)$, with $N$ a unit in
$\pi_0$.  It has three types of functoriality:
\begin{itemize}
\item inclusions $\Gamma < \Gamma'$ produce maps $\Tmf(\Gamma') \to \Tmf(\Gamma)$,
\item elements $g \in \GL_2(\mb Z/N)$ produce maps $\Tmf(\Gamma) \to
  \Tmf(g\Gamma g^{-1})$, and
\item projections $p\co \GL_2(\mb Z/NM) \to \GL_2(\mb Z/N)$
  produce natural equivalences $M^{-1} \Tmf(\Gamma) \to \Tmf(p^{-1}
  \Gamma)$.
\end{itemize}
These obey straightforward compatibility relations.  In addition, for
$K \lhd \Gamma < \GL_2(\mb Z/N)$ we find that the natural map
\[
\Tmf(\Gamma) \to \Tmf(K)^{h\Gamma/K}
\]
is an equivalence to the homotopy fixed-point spectrum.

(The reader who expects to see $\PSL_2$ rather than $\GL_2$ when
discussing modular curves should be aware that this is a difference
between the complex-analytic theory and the theory over $\mb Z$.
These modular curves may have several components when considered over
$\mb C$, and this may show up in the form of a larger ring of
constants.  Two modular curves may have the same set of points over
$\mb C$ because the group $\{\pm 1\}$ acts trivially on isomorphism
classes, but this action is nontrivial modular forms and on the
associated spectra.)

Our main theorems are consequences of our extending $\Tmf$ to a
fibrant presheaf of $E_\infty$ ring spectra on the log-\'etale site of
$\Mell$ (Theorem~\ref{thm:realmain}).  In rough:
\begin{itemize}
\item there is an assignment of $E_\infty$ ring spectra to certain
  generalized elliptic curves ${\cal E} \to X$ when the scheme
  $X$ is equipped with a compatible logarithmic structure;
\item this is functorial in certain diagrams
\[
\xymatrix{
{\cal E}' \ar[r] \ar[d] & {\cal E} \ar[d]\\
X' \ar[r] & X
}
\]
such that the map ${\cal E}' \to {\cal E} \times_X X'$
is an isomorphism of elliptic curves over $X'$;
\item this functor satisfies descent with respect to log-\'etale
  covers $\{U_\alpha \to X\}$ (or hypercovers), in the sense that the
  value on $X$ is the homotopy limit of the values on the \v Cech
  nerve; and
\item in the special case where ${\cal E} \to X$ comes from a
  Weierstrass curve over $\Spec(R)$, the associated spectrum realizes
  it by an even-periodic elliptic cohomology theory $E$.  (For more
  general elliptic curves on $\Spec(R)$, it is possible to show that
  the functor produces a weakly even-periodic object, and the formal
  group scheme $\Spf{} E^*(\mb{CP}^\infty)$ will come equipped with a
  natural isomorphism to the formal group of ${\cal E}$.)
\end{itemize}
The specific property of ${\cal E}/X$ needed is that the resulting
map $X \to \Mell$ classifying it must be log-\'etale. In particular,
any object \'etale over a modular curve satisfies this property, and
so the theorem simultaneously produces compatible presheaves of
$E_\infty$ ring spectra for all the modular curves.

The forms of $K$-theory from \cite[Appendix A]{tmforientation} play an
important role in our construction, and allow us to generalize the
main result of that paper.  A generalized elliptic curve ${\cal E} \to
X$ has a restriction to the ``cusps'' $X^c \subset X$.  The
restriction of ${\cal E}$ to $X^c$ is essentially a form of the
multiplicative group $\mb G_m$, and there is a corresponding form of
$K$-theory \cite{morava-forms}.  Our proof produces a
natural map of $E_\infty$ ring spectra from our functorial elliptic
cohomology theory built from $X$ to this form of $K$-theory built from
$X^c$ (Theorem~\ref{thm:thmcusp}).

We can then apply this to construct spectra $\tmf_1(3)$ and
$\tmf_0(3)$, connective $E_\infty$ ring spectra that realize
calculations carried out by Mahowald and Rezk
\cite{mahowald-rezk-level3}, together with a commutative diagram of
$E_\infty$ ring spectra
\[
\xymatrix{
\tmf_0(3) \ar[r] \ar[d] & ko[1/3] \ar[d] \\
\tmf_1(3) \ar[r] & ku[1/3]
}
\]
(Theorem~\ref{thm:tmfexists}). In particular, this $E_\infty$
connective version seems likely to coincide with one of their
conjectural models.  Carrying out calculations with further level
structures should be a very interesting and relatively accessible
problem.

There are several directions for further investigation and desirable
generalizations of the results of this paper.

One obvious missing component is the connection to elliptic genera.
The work of Ando, Hopkins, Rezk, and Strickland produced highly
multiplicative lifts of the sigma orientation and the
Atiyah-Bott-Shapiro orientation \cite{ando-hopkins-strickland-sigma,
  ando-hopkins-rezk-kotheory}, ultimately in the form of $E_\infty$
maps $MO\langle 8\rangle \to \Tmf$ and $MO\langle 4\rangle \to KO$.
\begin{question}
Does the Ochanine genus of Spin manifolds
\cite{ochanine-ellipticgenus} lift to an $E_\infty$ ring map $MSpin
\to \Tmf_0(2)$?
\end{question}
More generally, it would be very useful to know which objects
$\Tmf(\Gamma)$ (such as $\Tmf_1(3)$, which is complex orientable)
accept refinements of the sigma orientation to maps from $MSpin$,
$MSO$, $MU$, or others.

Rognes has recently developed a closely related concept of topological
logarithmic structures for applications in algebraic $K$-theory
\cite{rognes-logarithmic}.  The core construction in this paper is
built on a map of $E_\infty$ spaces from $\mathbb{N}$ to the
multiplicative monoid of $\Tmf$ after completion at the cusp, modeling
the logarithmic structure of $\Mell$ at the cusp.  This ultimately
imparts a topological logarithmic structure to the completion
$KO\pow{q}$.
\begin{question}
Does the topological logarithmic structure on $KO\pow{q}$ extend to a
topological logarithmic structure on $\Tmf$?
\end{question}
If so, it is natural to suspect that logarithmic obstruction theory is
a natural way to construct our presheaf.  There does not yet seem to be
an obstruction theory for maps between ring spectra with logarithmic
structures in the literature that is developed enough to carry out
this program, but this is almost entirely due to how recently
topological logarithmic structures have appeared.

Currently, $\Tmf$ is a functor defined on certain pairs of a scheme
together with an elliptic curve, with maps defined for change-of-base
and for isomorphisms of elliptic curves. However, separable isogenies
of elliptic curves are further maps that produce isomorphisms of 
formal group laws.
\begin{question}
Does the functor $\Tmf$ extend to a presheaf on the moduli object
classifying elliptic curves and separable isogenies?
\end{question}
After inverting $\ell$, this would allow the construction of two maps
$\Tmf \rightrightarrows \Tmf_0(\ell)$ classifying the two canonical
isogenous curves over $\Tmf_0(\ell)$, extend to the construction of
global versions of Behrens' $Q(\ell)$ spectra \cite{mark-buildings},
and allow an ``adelic'' formulation of the functoriality of $\Tmf$.
However, obstruction theory seems to be an inadequate tool for this
job.  For example, this would require constructing an action of $\mb
N^k$ on $N^{-1} \Tmf$, in the form of a commuting family of Adams
operations $[\ell]$ for primes $\ell$ dividing $N$.  Since this
obstruction theory is not in the category of $\Tmf$-algebras, the
relevant obstruction groups are not zero.  We hope that a consequence
of Lurie's constructive methods for associating spectra to
$p$-divisible groups (as employed in \cite{taf}) will be that the
smooth object $\TMF$ becomes functorial in separable isogenies, and
that the construction patching over the cusp in this paper will
inherit these isogeny operations from compatibility with the Adams
operations on complex $K$-theory.  We have chosen to write this paper
without appealing to Lurie's forthcoming work.

Some PEL Shimura varieties at higher heights, such as Picard
modular surfaces at height $3$, have compactifications (such as
Satake-Baily-Borel compactifications and smooth compactifications)
that generalize the compactification of the moduli of elliptic curves
\cite{faltings-chai}.
\begin{question}
Are there compactifications of other PEL Shimura varieties with
presheaves of spectra that extend the known presheaves of topological
automorphic forms on the interior?
\end{question}

Finally, the construction of the object $\tmf$ by connective cover
remains wholly unsatisfactory, and this is even more true when
considering level structure.  In an ideal world, $\tmf$ should be a
functor on a category of Weierstrass curves equipped with some form of
extra structure.  We await the enlightenment following discovery of
what exact form this structure should take.

\subsection{Outline of the method}

In order to minimize the amount of repeated effort, our proof is based
heavily on the tools used to construct topological modular forms in
\cite{mark-construct}.  The original method of Hopkins-Miller for the
construction of $\Tmf$ started with a construction away from the
supersingular locus, a construction at the supersingular locus, and an
argument in obstruction theory for patching the two together as
presheaves.  Though the cusp played no special role in their
construction, it will need to in this case. Our work follows a similar
paradigm to that of Hopkins-Miller: we start with their construction
away from the cusp, construct $\Tmf$ on the log site explicitly in a
formal neighborhood of the cusp, and then patch the two constructions
together.

In \S\ref{sec:loggeom} we give the necessary background on logarithmic
geometry that we need for this paper.  Since the objects of importance
to us will often be Deligne-Mumford stacks rather than schemes, we
need to express the theory in this generality.

In \S\ref{sec:elliptic} we have only two main goals.  The first goal
is to discuss the moduli objects in question.  Coarsely, they
parametrize log schemes equipped with an elliptic curve and some data
expressing compatibility of the logarithmic structure with the
cusps.  The second goal is to show that producing our desired
derived structure presheaf ${\cal O}$ is equivalent to defining
functorial elliptic cohomology theories on a much smaller category:
log rings equipped with suitable Weierstrass curves.  This requires
delving into some details about Grothendieck topologies.  Our approach
to this might be described as utilitarian.

In \S\ref{sec:homotopy-theory}, we assemble a number of fundamental
tools we will need from homotopy theory. The topics include elliptic
spectra, obstruction theory, rectification of diagrams where the
mapping spaces are homotopy discrete, homotopical versions of
sheafification, and background on $K$-theory and $\tmf$.

We may then begin to construct our presheaf ${\cal O}$ in
\S\ref{sec:construction}.

In \S\ref{sec:cusps}, we first use real $K$-theory to produce a
structure presheaf for log-\'etale maps to a formal
neighborhood $\mtate$ of the cusp, classifying forms of the Tate
curve. We start with Tate $K$-theory, an elliptic cohomology theory
with $\mathbb{Z}/2$-action called $K\pow{q}$ that was studied in
\cite{ando-hopkins-strickland-witten}.  We then extend it to a presheaf
for log-\'etale maps to $\mtate$ by direct construction. In this
section we also construct a natural map to forms of $K$-theory,
corresponding to evaluating at the cusps.

In \S\ref{sec:smooth}, we glue the $p$-completed structure presheaves
on the smooth locus and the cusps.  The basic construction at the
cusps starts with a map $\tmf \to KO\pow{q}$, an $E_\infty$ map
factoring the Witten genus, which is constructed in
Appendix~\ref{sec:witten}.  When we localize, it produces a map $\TMF
\to q^{-1} KO\pow{q}$, and degeneration of the Goerss-Hopkins
obstruction theory for $\Tmf$-algebras allows us to extend to a map of
presheaves.

Finally, \S\ref{sec:rational} uses an arithmetic square to glue the
$p$-complete constructions together with the rational constructions to
produce a global, integral, version.

This assembles all the pieces necessary to construct $\Tmf$
with level structure in \S\ref{sec:results}.

\subsection{Acknowledgments}

The authors would like to thank 
Matthew Ando,
Mark Behrens,
Andrew Blumberg,
Scott Carnahan,
Jordan Ellenberg,
Paul Goerss,
Mike Hopkins,
Nitu Kitchloo,
Michael Mandell,
Akhil Mathew,
Lennart Meier,
Niko Naumann,
William Messing,
Arthur Ogus,
Kyle Ormsby,
Charles Rezk,
Andrew Salch,
George Schaeffer, and
Vesna Stojanoska
for discussions related to this paper.  The anonymous referee of
\cite{tmforientation} also asked a critical question about
compatibility with $\mb Z/2$-actions, motivating our proof that
evaluation at the cusp is possible.

The ideas in this paper would not have existed without the Loen
conference ``$p$-Adic Geometry and Homotopy Theory'' introducing us to
logarithmic structures in 2009; the authors would like to thank the
participants there, as well as Clark Barwick and John Rognes for
organizing it.

This paper is written in dedication to Mark Mahowald.

\section{Logarithmic geometry}
\label{sec:loggeom}

\subsection{Logarithmic structures}

Our primary reference for logarithmic structures is
\cite{illusie-logoverview}, supplemented by \cite{kato-logarithmic,
  rognes-logarithmic, ogus-logbook}.  Since our ultimate objects of
study will be moduli of elliptic curves, we must also discuss this in
the stack case.

Throughout this paper we will assume that monoids have units, maps of
monoids are unital, and Deligne-Mumford stacks are separated.

\begin{defn}
A prelogarithmic structure on a scheme or Deligne-Mumford stack $X$
is an \'etale sheaf of commutative monoids $M_X$ on $X$, together
with a map $\alpha\co M_X \to {\cal O}_X$ of sheaves of commutative
monoids.  (Here ${\cal O}_X$ is viewed as a sheaf of monoids under
multiplication.)  We refer to $(X,M_X)$ as a prelog scheme or prelog
stack respectively.

A map of prelog schemes or prelog stacks $(Y,M_Y) \to (X,M_X)$ is a
map $Y \to X$ together with a commutative diagram
\[
\xymatrix{
M_X \ar[r] \ar[d] &
{\cal O}_X \ar[d] \\
f_* M_Y \ar[r] &
f_* {\cal O}_Y
}
\]
of sheaves of abelian monoids on $X$.
\end{defn}
\begin{defn}
A prelogarithmic structure $\alpha\co M_X \to {\cal O}_X$ is a
logarithmic structure if the map $\alpha$ restricts to an isomorphism
of sheaves $\alpha^{-1}({\cal O}_X^\times) \to {\cal O}_X^\times$.  In
this case we refer to $(X,M_X)$ as a log scheme or log stack.  A map
of log schemes or log stacks is a map of the underlying prelog
schemes or stacks.
\end{defn}
\begin{defn}
If $X$ has a prelogarithmic structure $\alpha\co M_X \to {\cal
O}_X$, the associated logarithmic structure $(M_X)^a$ is the
pushout of \'etale sheaves of commutative monoids
\[
\xymatrix{
  \alpha^{-1}({\cal O}_X^\times) \ar[r] \ar[d] & M_X \ar[d]\\
  {\cal O}_X^\times \ar[r] & (M_X)^a,
}
\]
equipped with the map $(M_X)^a \to {\cal O}_X$ determined by the
universal property.
\end{defn}

This logification functor $(X,M_X) \mapsto (X,(M_X)^a)$ is right
adjoint to the forgetful functor from log stacks to prelog stacks.

\begin{defn}
If $P$ is a commutative monoid with a map of commutative monoids $P
\to {\cal O}_X(X)$, we abuse notation by writing $P^a$ for the
logarithmic structure associated to the map from the constant sheaf
$P$ on $X$ to ${\cal O}_X$. The induced map $X \to
\Spec(\mb Z[P])$ then lifts to a map of log schemes.

In particular, if $\pi$ is a global section of ${\cal O}_X$, there is
a unique map of commutative monoids $\mb N \to {\cal O}_X(X)$ sending
the generator to $\pi$.  This is a prelog structure on $X$, and we
write $\langle \pi \rangle$ for the associated log structure
$(X,\langle \pi \rangle)$.
\end{defn}
\begin{defn}
\label{def:directimage}
Suppose $j\co U \subset X$ is an open subobject and $U$ is equipped
with a logarithmic structure $M_U$.  Let $j_\sharp M_U$ be the
pullback of the diagram of sheaves of commutative monoids on $X$:
\[
\xymatrix{
j_\sharp M_U \ar[d] \ar[r] & {\cal O}_X \ar[d] \\
j_* M_U \ar[r] & j_* {\cal O}_U
}
\]
The direct image logarithmic structure on $X$ is the associated
logarithmic structure $(j_\sharp M_U)^a$.
\end{defn}

In the particular case where $M_U$ is the trivial log structure
$(\mathcal{O}_U^\times \to \mathcal{O}_U)$ on $U$, this logarithmic
structure consists of the submonoid of ${\cal 
O}_X$ of functions whose restriction to $U$ is invertible.

\begin{defn}
\label{def:logproperties}
We say that $(X,M_X)$ is:
\begin{itemize}
\item integral if each group-completion map on stalks $M_{X,x} \to
  (M_{X,x})^{gp}$ is an injective map of monoids; 
\item quasicoherent if there exists a cover of $X$ by \'etale
  neighborhoods $U$ with maps of commutative monoids $P_U \to M_X(U)$
  such that the induced map $(P_U)^a \to M_X|_U$ is an isomorphism (we
  refer to such a pair $(U, P_U)$ as a chart for the logarithmic
  structure on $X$);
\item coherent if there exists a cover of $X$ by charts
  $(U,P_U)$ such that $P_U$ is finitely generated;
\item fine if it is integral and coherent;
\item saturated if it is integral, and if $M_X$ contains all
  sections $s$ of $(M_X)^{gp}$ which have a (positive integer)
  multiple in $M_X$.
\end{itemize}
\end{defn}

In this paper, we restrict attention to fine and saturated logarithmic
structures.  We note that $(X,M_X)$ is fine and saturated if and only
if there is a cover of $X$ by charts $(U,P_U)$ such that the monoid
$P_U$ itself is fine and saturated: $P_U$ is finitely generated, the
group completion map $P_U \to (P_U)^{gp}$ is injective, and an element
in $(P_U)^{gp}$ has a positive multiple in the image of $P_U$ if
and only if it is in $P_U$.

We write $\LSc$ for the category of fine and saturated log schemes,
with associated forgetful functor $\LSc \to \Sch$ to the category of
schemes.  This functor has a fully faithful right adjoint $\Sch \to
\LSc$, which takes a scheme $X$ and gives it the trivial log structure
$(X,{\cal O}_X^\times)$.  We will implicitly view $\Sch$ as a full
subcategory of $\LSc$ via this map.  Similarly, we have a full
subcategory $\ALS$ of affine log schemes, consisting of objects whose
underlying scheme is affine.

\subsection{Divisors with normal crossings}

We recall the definition of a normal crossings divisor
\cite[3.1.5]{sga5} and give a definition in the stack context.
\begin{defn}
\label{def:divisors}
Let $X$ be a scheme.  A strict normal crossings divisor on $X$ is a
divisor $D$ such that there exist sections $f_1,\ldots,f_r \in {\cal
  O}_X(X)$ with the properties that
\begin{enumerate}
\item $X$ is regular at all points in the support of $D$,
\item $D = \sum {\rm div}(f_i)$, and
\item for all $x$ in the support of $D$, the
  set of elements $f_i$ in the maximal ideal at $x$ form part of a
  regular sequence of parameters for the local ring ${\cal
    O}_{X,x}$.
\end{enumerate}
A normal crossings divisor on $X$ is one for which there exists an
\'etale cover by maps $U_\alpha \to X$ such that $D|_{U_\alpha}$ is
a strict normal crossings divisor for all $\alpha$.

A normal crossings divisor is smooth if we may take $r=1$ in the above
definition.
\end{defn}

\begin{defn}
Let $X$ be a Deligne-Mumford stack.  A normal crossings divisor $D$ on
$X$ is a map associating a normal crossings divisor $D_Y$ to each
scheme $Y$ \'etale over $X$, compatible with pullbacks for morphisms
over $X$.  The complement $X \setminus D$ is the open substack whose
restriction to each such $Y$ is $Y \setminus D_Y$.  The divisor $D$ is
smooth if all its restrictions $D_Y$ are smooth.
\end{defn}

\begin{rmk}
Suppose $\{U_\alpha \to X\}$ is an \'etale cover of $X$ by schemes.
The data of a normal crossings divisor (resp. smooth divisor) on $X$ is
equivalent to normal crossings divisors (resp. smooth divisors) $D_\alpha$
on $U_\alpha$ such that, on any intersection $U_\alpha \times_X
U_\beta$, we have $p_1^* D_\alpha = p_2^* D_\beta$.  In particular,
for stacks of the form $[Y /\!/ G]$ for $Y$ a scheme with $G$-action,
this is equivalent to a choice of $G$-invariant divisor on $Y$.
\end{rmk}

\begin{rmk}
\label{rmk:sectiondivisor}
Suppose ${\cal L}$ is a line bundle on $X$ with a global section $s
\in \Gamma(X,{\cal L})$ which induces an embedding ${\cal O}_X \to
{\cal L}$. Tensoring with ${\cal L}^{-1}$, we obtain an ideal sheaf
${\cal I} \subset {\cal O}_X$ which defines a compatible family of
divisors (specifically, the divisor associated to the section
$s$).
\end{rmk}

\begin{defn}
If $D$ is a normal crossings divisor on $X$, then the associated log
structure $M_D \to {\cal O}_X$ is the direct image of the trivial log
structure on $X \setminus D$.
\end{defn}

\begin{prop}
\label{prop:ncdfs}
If $D$ is a normal crossings divisor on $X$, then the associated log
stack $(X,M_D)$ is fine and saturated.
\end{prop}

\begin{proof}
By definition, for sufficiently small \'etale opens $U \to X$ we have
a sequence of sections $f_1,\ldots,f_r \in {\cal O}_X(U)$ generating
the divisor and forming part of a regular sequence of parameters for
the local rings in the support of $D$.  These determine a map $\mb N^r
\to M_D \subset {\cal O}_X(U)$ from the free commutative monoid on 
$r$ generators, which is fine and saturated.  Regularity implies that
any function invertible away from $D$ is locally and uniquely a unit
times a product of the $f_i$, and so the resulting logarithmic
structure is isomorphic to $M_D$.
\end{proof}

\subsection{Log-\'etale maps}

We recall from Definition~\ref{def:logproperties} that a chart $(U,P)$
of a log stack $(X,M_X)$ is a particular type of map $(U,P) \to
(X,M_X)$ from a prelog scheme to a log stack: an \'etale map $U \to X$
and a commutative monoid $P$ mapping to ${\cal O}_X(U)$ such that the
associated map $P^a \to M_X|_U$ is an isomorphism.  Similarly, a chart
of a map $f\co (Y,M_Y) \to (X,M_X)$ is a commutative diagram
\[
\xymatrix{
(V,Q) \ar[r] \ar[d] & (Y,M_Y) \ar[d]\\
(U,P) \ar[r] & (X,M_X)
}
\]
of prelog stacks whose horizontal arrows are charts; these are
determined by maps $V \to U$ and $P \to Q$ making certain diagrams
commute.  If both logarithmic structures are fine, any map $f$ has an
\'etale cover by charts.

\begin{defn}[{\cite[3.3]{kato-logarithmic}}]
A map $(Y,M_Y) \to (X,M_X)$ is log-\'etale if the underlying map $Y
\to X$ is locally of finite presentation, and there exists a cover
of $f$ by charts $(V,Q) \to (U,P)$ such that
\begin{enumerate}
\item the map $P \to Q$ becomes an injection $P^{gp} \to Q^{gp}$ whose
  cokernel has finite order invertible on $U$, and
\item the induced map $V \to U \times_{\Spec(\mb Z[P])} \Spec(\mb
  Z[Q])$ is \'etale.
\end{enumerate}
\end{defn}

\begin{rmk}
As with ordinary \'etale maps, these maps can also be defined in
terms of an infinitesimal lifting criterion or in terms of inducing
an isomorphism on logarithmic cotangent complexes \cite[3.3,
3.12]{kato-logarithmic}.
\end{rmk}

\begin{rmk}
In particular, any map $P \to Q$ which induces an isomorphism $P^{gp}
\to Q^{gp}$ produces log-\'etale maps $(\Spec(k[Q]), Q^a) \to
(\Spec(k[P]), P^a)$. Even in the case of fine and saturated monoids,
these can produce a wide variety of phenomena on the underlying
schemes.
\end{rmk}

The Kummer log-\'etale maps are particularly well-behaved log-\'etale
maps.
\begin{defn}[{\cite[1.6]{illusie-logoverview}}]
A map $h\co P \to Q$ of commutative monoids is of Kummer type if it is
injective, and for any $a \in Q$ there exists an $n \geq 1$ such that
$a^n$ is in the image of $h$.

A map $(Y,M_Y) \to (X,M_X)$ is Kummer log-\'etale if it is log-\'etale
and, at each point $y$ of $Y$, the induced map on stalks
\[
(M_X/{\cal O}_X^*)_{X,\overline{f(y)}} \to (M_Y/{\cal
    O}_Y^*)_{Y,\overline y}
\]
is of Kummer type.
\end{defn}

\begin{rmk}
\label{rmk:kummeretale}
In particular, if $f\co (Y,M_Y) \to (X,M_X)$ can be covered by
log-\'etale charts $(V,Q) \to (U,P)$ such that the map of monoids $P
\to Q$ is of Kummer type, the map $f$ is Kummer log-\'etale.
\end{rmk}

\begin{prop}
\label{prop:kummer}
Suppose $(X,M_X)$ is a log stack determined by a smooth divisor $D$ on
$X$.  Then any log-\'etale map $(Y,M_Y) \to (X,M_X)$ is flat and Kummer
log-\'etale, and the log structure on $Y$ is determined by a smooth
divisor on $Y$ over $D$.
\end{prop}

\begin{proof}
It suffices to work locally in the \'etale topology, so without loss
of generality it suffices to prove this on a coordinate chart of the
form $(U,\langle \pi \rangle) \to (X,M_X)$ for a scheme $U$, where we
view $\pi$ as a map $\mb N \to {\cal O}_X(U)$.

Consider the following data:
\begin{itemize}
\item a finitely generated commutative monoid $Q$ which is integral
  and saturated;
\item a map of monoids $\mb N \to Q$ such that, on group completion,
  the map $\mb Z \to Q^{gp}$ is an injection with cokernel of order $n
  < \infty$; and
\item an \'etale map of schemes $V \to U \times_{\Spec(\mb Z[\mb N])} 
  \Spec(\mb Z[Q][1/n])$.
\end{itemize}
In this situation, the monoid $Q$ generates the log structure $Q^a =
M_V$ on $V$.  The associated map $(V,M_V) \to (U,\langle\pi\rangle)$
is called a standard log-\'etale map, and it is Kummer log-\'etale by
definition.

By \cite[3.5]{kato-logarithmic}, any log-\'etale map $(Y,M_Y)
\to (U,\langle \pi\rangle)$ is, \'etale-locally on $Y$, a standard
log-\'etale map.

The hypotheses on $Q$ imply that $Q^{gp} \cong \mb Z \times F$ where
$|F|$ divides $n$, and the saturation property implies that either $Q
= \mb N \times F$ or $Q = \mb Z \times F$; both of these generate the
same log structure as the submonoid $\mb N \times 0$.

We obtain a factorization of the map $V \to U$ as
\[
V \to U \times_{\Spec(\mb Z[\mb N])}\Spec(\mb Z[Q][1/n]) \to U \times \Spec(\mb Z[F][1/n]) \to U.
\]
The first and last maps are \'etale, and so it suffices to show that
the middle map is flat and that the log structure is generated by a
smooth divisor.  Letting $Z = U \times \Spec(\mb Z[F][1/n])$, the
center map is of the form
\[
Z \times_{\Spec(\mb Z[\pi])} \Spec(\mb Z[(u\pi)^{1/d}]) \to Z
\]
for some unit $u \in F$ and $d$ dividing $n$.  This map is, in
particular, finite flat.  The log structure on the cover is generated
by $(u\pi)^{1/d}$.  Since $\pi$ describes a smooth divisor
on $Z$ the element $(u\pi)^{1/d}$ determines a smooth divisor on the
cover: at any point of the support of $D$, any regular sequence on $Z$
containing $\pi$ lifts to a regular sequence on the cover once we
replace $\pi$ with $(u\pi)^{1/d}$.
\end{proof}

In particular, this shows that there is a cofinal collection of
log-\'etale covers of $(U,\langle \pi\rangle)$ consisting of \'etale
covers of
\[
(Y \times_{\Spec(\mb Z[\pi])} \Spec(\mb Z[\pi^{1/n}, 1/n]), \langle
\pi^{1/n}\rangle).
\]
In \S\ref{sec:cusps} we will make use of the special case of a
power series ring.

\begin{cor}
\label{cor:logetalecusp}
Any log-\'etale map $(\mb Z\pow{q},\langle q\rangle) \to (R,M)$, where
$R$ is a formal $\mb Z\pow{q}$-algebra, has an \'etale cover by maps 
\[
(\mb Z\pow{q}, \langle q \rangle) \to 
(\mb Z\pow{q^{1/m}}, \langle q^{1/m} \rangle) \to
(S, \langle q^{1/m}\rangle),
\]
where $S$ is an \'etale formal $\mb Z\pow{q^{1/m}}$ algebra with $m$
invertible.
\end{cor}

\begin{rmk}
We note that the integer $m$ is recovered as a ramification index of
the ideal $(q)$, and that the \'etale extension $S$ of $\mb
Z\pow{q^{1/m}}$ is uniquely of the form $\bar S\pow{q^{1/m}}$, where
$\bar S$ is the \'etale extension $S/(q^{1/m})$ of $\mb Z[1/m]$.  In
other words, such an extension is determined up to isomorphism by the
ramification index and the residue extension, but isomorphisms may
differ by a map $q^{1/m} \mapsto \zeta q^{1/m}$.
\end{rmk}

\begin{rmk}
\label{rmk:ncd-kummer}
All of the conclusions of Proposition~\ref{prop:kummer}, except for
the final one, remain true if we replace smooth divisors by normal
crossings divisors.  It is not necessarily the case that a normal
crossings divisor locally lifts to a normal crossings divisor.
\end{rmk}

\subsection{Logarithmic topologies}

The Kummer log-\'etale maps give rise to a Kummer log-\'etale topology
\cite[\S2]{illusie-logoverview}.  Specifically, it is a Grothendieck
topology generated by covers which are collections $\{(U_\beta,
M_{U_\beta}) \to (X,M_X)\}$ of Kummer log-\'etale maps such that the
underlying scheme $X$ is the union of the images of $U_\beta$.
The representable functors are sheaves for this topology, and
quasicoherent sheaves on $X$ automatically extend to sheaves on the
Kummer log-\'etale site of $X$ \cite{niziol-logschemes-I}.

Proposition~\ref{prop:kummer} has the following consequence which, as
in Remark~\ref{rmk:ncd-kummer}, generalizes to the case of a normal
crossings divisor.
\begin{prop}
Suppose $(X,M_X)$ is a log stack defined by a smooth divisor.  Then
the log-\'etale and Kummer log-\'etale topologies are equivalent.
\end{prop}

The Kummer site has the advantage that we can obtain control on
cohomology.

\begin{prop}[{\cite[3.27]{niziol-logschemes-I}}]
\label{prop:serrevanishing}
Suppose $(X,M_X)$ is an affine log scheme with a sheaf ${\cal F}$ on
its Kummer log-\'etale site which is locally quasicoherent.  Then the
cohomology of $(X,M_X)$ with coefficients in ${\cal F}$ vanishes in
positive degrees.
\end{prop}

\begin{rmk}
Quasicoherence is not a local property in the log-\'etale topology.
Log-\'etale maps are very often not flat, which means that log-\'etale
covers might not satisfy common types of descent.
\end{rmk}

From a standard hypercover argument relating derived functor
cohomology to coherent cohomology, we obtain the following.

\begin{cor}
\label{cor:quasicoherentcohomology}
Suppose $(X,M_X)$ is a log stack with a quasicoherent sheaf ${\cal F}$
on $X$.  Then the natural map from Zariski cohomology of $X$ to the
Kummer log-\'etale cohomology of $(X,M_X)$ with coefficients in ${\cal
  F}$ is an isomorphism.
\end{cor}

\subsection{Log-\'etale covering maps}

The following is analogous to the definition of a covering space.

\begin{defn}[cf. {\cite[3.1]{illusie-logoverview}}]
A map $(Y,M_Y) \to (X,M_X)$ of log stacks is a Kummer log-\'etale
covering map if there is a cover of $(X,M_X)$ by log schemes $(X_\beta,
M_\beta)$, in the Kummer log-\'etale topology, such that each pullback of
$(Y,M_Y)$ in the category of fine and saturated log stacks is
isomorphic to a finite union of copies of $X_\beta$.

Suppose a finite group $G$ acts on $(Y,M_Y)$ with quotient $(X,M_X)$.
The map $(Y,M_Y) \to (X,M_X)$ is a (Kummer log-\'etale) Galois
covering map if it is a Kummer log-\'etale covering map $(Y,M_Y) \to
(X,M_X)$, and the shearing map
\[
G \times (Y,M_Y) \to (Y,M_Y) \times_{(X,M_X)} (Y,M_Y),
\]
expressed by $(g,y) \mapsto (gy,y)$, is an isomorphism. Here the
source is the log scheme $\coprod_{g \in G} (Y,M_Y)$.
\end{defn}

In the case of a normal crossings divisor, we have the following.
\begin{prop}
\label{prop:coveringequiv}
Suppose $(X,M_X)$ is a log stack associated to a normal crossings
divisor $D$, and that $X$ is smooth away from $D$.  Then there is an
equivalence of categories between Kummer log-\'etale covering maps
$(Y,M_Y) \to (X,M_X)$ and \'etale covering maps of $X \setminus D$
which are tamely ramified over $D$, given by
\[
(Y, M_Y) \mapsto Y \times_X (X \setminus D).
\]
The inverse sends an \'etale cover $Y' \to X \setminus D$ to the
normalization of $X$ in $Y'$.

This induces an equivalence between Galois covers with group $G$ and
Kummer log-\'etale Galois covers of $Y \setminus D$ with group $G$
which are tamely ramified over $D$.
\end{prop}

\begin{proof}
This combines results from \cite[7.3(b), 7.6]{illusie-logoverview}.
On any log scheme which is an \'etale open of $(X,M_X)$, the given
maps establish an equivalence of categories.  As \'etale covers and
log-\'etale covers satisfy \'etale descent, and tame ramification is
an \'etale-local property, the result follows.
\end{proof}

\section{The elliptic moduli}
\label{sec:elliptic}

\subsection{The logarithmic moduli of elliptic curves}

Let $\Mell$ be the compactified moduli of generalized elliptic curves,
which is a Deligne-Mumford stack.  The stack $\Mell$ contains a cusp
divisor classifying elliptic curves with nodal singularities, and the
complement is an open substack $\mell$ classifying smooth elliptic
curves.

\begin{defn}
The log stack $\Mlog$ is the stack $\Mell$ equipped with the direct
image log structure from $\mell$ (\ref{def:directimage}).
\end{defn}

We recall that the category $\Sch/\Mell$ is a stack in the fpqc
topology on $\Sch$: maps $X \to \Mell$ classify generalized elliptic
curves on $X$.  In particular, $\Sch/\Mell$ is a category whose the
objects are pairs $(X,{\cal E})$ consisting of a scheme $X$ and a
generalized elliptic curve ${\cal E} \to X$, and the morphisms are
pullback diagrams
\[
\xymatrix{
{\cal E}' \ar[r] \ar[d] & {\cal E} \ar[d] \\
X' \ar[r] & X
}
\]
respecting the chosen section of the elliptic curve. The functor which
forgets ${\cal E}$ makes this into a category fibered in groupoids
over $\Sch$.

We denote by $\omega / \Mell$ the sheaf of invariant differentials on
the universal elliptic curve; global sections of $\omega^k$ are
modular forms of weight $k$. The group $\Gamma(\Mell, \omega^{12})$ is
freely generated over $\mb Z$ by the modular forms $c_4^3$ and
$\Delta$. These do not vanish simultaneously on $\Mell$ and so we can
make the following definition.

\begin{defn}
The $j$-invariant
\[
j\co \Mell \to \mb P^1
\]
is the map given in homogeneous coordinates by
\[
{\cal E} \mapsto [c_4^3({\cal E}):\Delta({\cal E})].
\]
\end{defn}

We first need a coordinate chart on $\Mell$ near the cusps.

\begin{lem}
\label{lem:cuspcoords}
The generalized elliptic curve
\[
y^2 + xy = x^3 - \tfrac{36}{j-1728}x - \tfrac{1}{j-1728}
\]
defines an \'etale map ${\cal C}\co \mb P^1 \setminus \{0,1728\} \to
\Mell$, covering the complement of the $j$-invariants $0$ and $1728$.
\end{lem}

\begin{proof}
This curve has the prescribed $j$-invariant $j$
\cite[III.1.4.c]{silverman-aec}, and on $\Spec(\mb Z[j^{-1}])$ it has
only nodal singularities except when $(1728j^{-1}-1) = 0$.  This curve
then describes a map from $\Spec(\mb Z[j^{-1}, (1728j^{-1} -
1)^{-1}])$ to $\Mell$ producing generalized curves of all
$j$-invariants other than $0$ and $1728$.

Let $U \subset \Mell$ be the preimage of $\mb P^1 \setminus \{0,
1728\})$ under $j$. By \cite[8.4.3]{katz-mazur}, there are no
automorphisms of generalized elliptic curves classified by maps to $U$
other than the identity and the negation automorphism; the
automorphism group of the universal elliptic curve on $U$ is the
constant group $\mb Z/2$. As a result, for any elliptic curve ${\cal
  E}\co X \to U$, the fiber product $X \times_U (\mb P^1 \setminus
\{0,1728\})$ classifying isomorphisms between ${\cal E}$ and the
pullback of ${\cal C}$ is a principal $\mb Z/2$-torsor ${\rm
  Iso}_{X}({\cal E}, {\cal C}(j({\cal E}))$. In particular, it is
\'etale over $X$. The map classifying ${\cal C}$ is therefore
\'etale and there is an equivalence of stacks
\[
U \cong B\mb Z/2 \times (\mb P^1 \setminus \{0,1728\}).\qedhere
\]
\end{proof}

As in Remark~\ref{rmk:sectiondivisor}, the section $\Delta$ of
$\omega^{12}$ determines a divisor: the cusp of $\Mell$.
\begin{cor}
\label{cor:smcd}
The cusp divisor of $\Mell$ is a smooth divisor
(\ref{def:divisors}).
\end{cor}

\begin{proof}
In the coordinates of Lemma~\ref{lem:cuspcoords}, the cusp divisor is
the smooth divisor described by $j^{-1}$ on an open subset of
$\Spec(\mb Z[j^{-1}])$.
\end{proof}

Proposition~\ref{prop:ncdfs} then has the following consequence.
\begin{prop}
The log stack $\Mlog$ is fine and saturated.
\end{prop}

\begin{defn}
The log scheme $\mb P^1_{\log}$ is the log structure on $\mb P^1$
defined by the divisor $(\infty)$.
\end{defn}

Explicitly, the monoid sheaf $M \subset {\cal O}_{\mb P^1}$
consists of functions whose restriction to $\mb P^1\setminus\{\infty\}
= \Spec(\mb Z[j])$ is invertible.

\begin{rmk}
For any integer $z$, viewed as an integral point $z \in \mb A^1$,
there is an isomorphism of log schemes
\[
\mb P^1_{log} \setminus \{z\} \cong (\Spec(\mb
Z[(j-z)^{-1}]), \langle(j-z)^{-1}\rangle).
\]
\end{rmk}

\begin{prop}
The log stack $\Mlog$ is the pullback
\[
\xymatrix{
\Mlog \ar[r] \ar[d]_{j_{log}} & \Mell \ar[d]^j\\
\mb P^1_{log} \ar[r] & \mb P^1.
}
\]
In particular, a log scheme over $\Mlog$ consists of a log scheme
$(X,M_X)$, a generalized elliptic curve ${\cal E}$ on $X$, and a lift
of the $j$-invariant $j({\cal E})\co X \to \mb P^1$ to a map of log
schemes $j_{log}({\cal E})\co (X,M_X) \to \mb P^1_{log}$.
\end{prop}

\begin{proof}
As the logarithmic structure is the trivial logarithmic structure on
the complement of the cusps, it suffices to use the coordinate chart
of Lemma~\ref{lem:cuspcoords} and show that there is a pullback
diagram of fine and saturated log schemes
\[
\xymatrix{
(\Spec(\mb Z[j^{-1}, (1728j^{-1} - 1)^{-1}]), \langle j^{-1} \rangle)
\ar[r] \ar[d] & 
\Spec(\mb Z[j^{-1}, (1728j^{-1} - 1)^{-1}]) \ar[d] \\
\mb P^1_{\log} \ar[r] &
\mb P^1.
}
\]
However, this is precisely the inclusion of a coordinate chart on the
log scheme $\mb P^1_{log}$.
\end{proof}

\subsection{Weierstrass curves}
\label{sec:weierstrass-curves}

We first recall that there is a ring 
\[
A = \mb Z[a_1,a_2,a_3,a_4,a_6]
\]
parametrizing Weierstrass curves of the form
\[
y^2 + a_1xy + a_3 y = x^3 + a_2 x^2 + a_4 x + a_6.
\]
This ring $A$ contains modular quantities $c_4$, $c_6$, and $\Delta$
\cite[III.1]{silverman-aec}, and the complement of the common
vanishing locus of $c_4$ and $\Delta$ is an open subscheme
\begin{equation}
  \label{eq:openweierstrass}
\Spec(A)^o = \Spec(c_4^{-1} A) \cup \Spec(\Delta^{-1} A)
\end{equation}
parametrizing generalized elliptic curves in Weierstrass form.    In
particular, there is a universal generalized elliptic curve ${\cal E}
\to \Spec(A)^o$.  This determines a map from $\Spec(A)^o$ to the
Deligne-Mumford stack $\Mell$ which is a smooth cover.

Similarly, there is a universal scheme parametrizing pairs of
isomorphic Weierstrass curves.  Let
\[
  \Gamma = A[r,s,t,\lambda^{\pm 1}]
\]
be the ring parametrizing the change-of-coordinates $y \mapsto
\lambda^3 y + r x + s$, $x \mapsto \lambda^2 x + t$ on Weierstrass
curves.  The pair $(A,\Gamma)$ forms a Hopf algebroid with an
invariant ideal $(c_4,\Delta)$.  Defining 
\[
\Spec(\Gamma)^o = \Spec(c_4^{-1} \Gamma) \cup \Spec(\Delta^{-1} \Gamma),
\]
we obtain a groupoid $(\Spec(A)^o, \Spec(\Gamma)^o)$ in schemes that
maps naturally to $\Mell$.

As $\Spec(\Gamma)^o$ parametrizes pairs of generalized elliptic curves
in Weierstrass form with a chosen isomorphism between them, we have an
equivalence to the pullback
\[
\Spec(\Gamma)^o \longoverto^\sim \Spec(A)^o \times_{\Mell} \Spec(A)^o.
\]

For Weierstrass curves, we note that the identity $j^{-1} = \Delta
c_4^{-3}$ and the constraint that the units of $M_X$ map
isomorphically to ${\cal O}_X^\times$ imply that a factorization
$(X,M_X) \to \mb P^1_{\log} \to \mb P^1$ of the $j$-invariant is
equivalent to a lift of the element $\Delta \in {\cal O}_X(c_4^{-1}
X)$ to a section of $M_X(c_4^{-1}X)$.  We will casually refer to this
as a lift of the elliptic discriminant to $M_X$.

The divisor defined by the vanishing of $\Delta$ determines a log
structure on $\Spec(A)^o$.  The fiber product $\Spec(A)^o \times_{\mb
P^1} \mb P^1_{log}$ is, in fact, the log scheme
\[
U_{log} = (\Spec(A)^o, \langle \Delta \rangle).
\]
This classifies the universal Weierstrass curve in log schemes.

The fiber product of $(\Spec(A)^o,\langle \Delta \rangle)$ with itself
over $\Mlog$ is then, by invariance of $\Delta$ up to unit,
\[
R_{log} = (\Spec(\Gamma)^o, \langle \Delta \rangle).
\]
The pair $(U_{log}, R_{log})$ form a smooth groupoid object in $\LSc$
that parametrizes the groupoid of Weierstrass curves with lifts of the
log structure.

\begin{prop}
The natural map $(U_{log},R_{log}) \to \Mlog$ of groupoids induces an
equivalence of stacks in the log-\'etale or Kummer log-\'etale
topologies.
\end{prop}

\begin{proof}
The induced map of groupoids is fully faithful. In addition, any map
$X \to \Mlog$ can be covered by maps which lift to $U_{log}$: in fact,
any elliptic curve on a scheme is isomorphic to a Weierstrass curve
locally in the Zariski topology.
\end{proof}

We remark that $\Mlog$ is merely a prestack, rather than a stack, in
the topology on $\LSc$ because elliptic curves do not obviously
satisfy descent for log-\'etale covers.

\begin{prop}
\label{prop:logdiagonal}
The map $U_{log} \to \Mlog$ is representable and a smooth cover.
\end{prop}

\begin{proof}
Given $(X,M_X) \to \Mlog \cong \Mell \times_{\mb P^1} \mb P^1_{log}$, the
pullback is
\[
(X,M_X) \times_{\Mell \times_{\mb P^1} \mb P^1_{log}} (\Spec(A)^o
\times_{\mb P^1} \mb P^1_{log}) \cong (X,M_X) \times_{\Mell} \Spec(A)^o.
\]
In particular, this follows from the fact that the map $\Spec(A)^o \to
\Mell$ is representable and smooth.
\end{proof}

As a result, within the Grothendieck topology on $\LSc$, $(U_{log},
R_{log})$ gives a presentation of the same stack as $\Mlog$.

\begin{rmk}
\label{rmk:etalecover}
This smooth cover can be refined to a Kummer log-\'etale cover.  Away
from the cusps, for example, there are {\em schemes} ${\cal
M}_1(4)[1/2]$ and ${\cal M}_1(3)[1/3]$.  These parametrize,
respectively, smooth elliptic curves with a chosen $4$-torsion point
away from the prime $2$, and smooth elliptic curves with a chosen
$3$-torsion point away from the prime $3$ \cite{katz-mazur}.
The coordinate chart of Lemma~\ref{lem:cuspcoords} gives an \'etale
cover over the cusps.
\end{rmk}

The following result will be useful in understanding the cohomology
rings of objects which are honestly \'etale over $\Mell$ (compare
\cite[construction of Diagram 9.2]{mark-construct}).

\begin{prop}
\label{prop:rationaletale}
If a map $\Spec(R) \to \Mell$ classifying an elliptic curve ${\cal
  E}$ is \'etale, then the map of graded rings
\[
\bigoplus_{k \in \mb Z} H^0(\Mell, \omega^k) \otimes \mb Q \to
\bigoplus_{k \in \mb Z} H^0(\Spec(R), \omega^k) \otimes \mb Q
\]
is \'etale.
\end{prop}

\begin{proof}
Let $\Spec(\mb Q[c_4,c_6])^o$ denote the open complement of the ideal
defined by $(c_4,\Delta)$ in $\Spec(\mb Q[c_4,c_6])$. We can form the
pullback in the diagram
\[
\xymatrix{
Y \ar@{.>}[r] \ar@{.>}[d] &
\Spec(\mb Q[c_4,c_6])^o \ar[r] \ar[d]^{\cal D} &
\Spec(\mb Q[c_4,c_6]) \\
\Spec(R) \ar[r]_{\cal E} &
\Mell.
}
\]
Here the center vertical map classifies the curve ${\cal D}$ given by
the Weierstrass equation $y^2 = x^3 - \tfrac{c_4}{48}
x - \tfrac{c_6}{864}$, which induces the isomorphism $\bigoplus H^0(\Mell,
\omega^k)\otimes \mb Q \to \mb Q[c_4,c_6]$. As the map classifying
${\cal E}$ is \'etale, so is the composite upper map $Y \to \Spec(\mb
Q[c_4,c_6])$. Moreover, this map is equivariant for the action of $\mb
G_m$ by $c_4 \mapsto \lambda^4 c_4, c_6 \mapsto \lambda^6 c_6$.

However, the pullback $Y$ is the universal object over $\Spec(R)$
which is both rational and where the curve ${\cal E}/R$ has a chosen
isomorphism to ${\cal D}$. As ${\cal D}$ has a chosen invariant
$1$-form, so does the associated elliptic curve on $Y$. However, away
from the primes $2$ and $3$ any curve ${\cal E}$ with a chosen
nowhere-vanishing $1$-form has a unique isomorphism to a curve $y^2 =
x^3 + px + q$ preserving the $1$-form.

Therefore, $Y \to \Spec(R)$ is the object classifying choices of
nowhere vanishing $1$-form on ${\cal E}/(R\otimes \mb Q)$, which is
$\Spec$ of the ring $\bigoplus_{k \in \mb Z} H^0(R,\omega^k)\otimes
\mb Q$. As a result, the map of $\mb G_m$-equivariant (i.e. graded)
rings $\mb Q[c_4,c_6] \to \bigoplus_{k \in \mb Z} H^0(R,\omega^k)
\otimes \mb Q$ is \'etale.
\end{proof}

\subsection{Log-\'etale objects over $\Mlog$}
\label{sec:log-etale-objects}

Proposition~\ref{prop:logdiagonal} allows us to study log-\'etale
objects over $\Mlog$.  As the log structure is determined by a smooth
divisor by Corollary~\ref{cor:smcd}, the log-\'etale covers determine
an equivalent topology to the Kummer log-\'etale topology.  It
suffices to check that a map $(X,M_X) \to \Mlog$ is log-\'etale on the
cover defined by the moduli of Weierstrass
curves~(\ref{eq:openweierstrass}).  (Alternatively, one can check that
the restriction to $\Mell$ is \'etale, and check that it is
log-\'etale on the chart of Lemma~\ref{lem:cuspcoords}.)

For any log stack $(X,M_X)$ over $\Mlog$ classifying an elliptic curve
${\cal E}/X$, form the pullback
\[
\xymatrix{
(X,M_X)\times_{\Mlog} U_{log}  \ar[r]^-p \ar[d] & U_{log} \ar[d]\\
(X,M_X) \ar[r] & \Mlog.
}
\]
This fiber product is universal among log stacks $(Y,M_Y)$ equipped
with a map $f\co (Y,M_Y) \to (X,M_X)$ and an isomorphism of $f^*({\cal
E})$ with a Weierstrass curve.  The object $(X,M_X)$ is (Kummer)
log-\'etale over $\Mlog$ if and only if the map $p$ of log stacks is
(Kummer) log-\'etale.

In particular, if the elliptic curve on $X$ is smooth, then $\Delta$
is invertible and there are no restrictions on the logarithmic
structure.  For such a map to make $(X,M_X)$ log-\'etale over $\Mlog$,
the logarithmic structure on $X$ must be trivial.

\begin{defn}
The small (Kummer) log-\'etale site of $\Mlog$ is the category of log
schemes $(X,M_X)$ equipped with a (Kummer) log-\'etale map $(X,M_X) \to
\Mlog$, with maps being the (Kummer) log-\'etale maps over $\Mlog$.
\end{defn}
The classical \'etale site of $\Mell$ has a fully faithful embedding
into the small log-\'etale site of $\Mlog$, and so the latter is
strictly an enlargement.

The affine examples of log stacks are {\em log rings}, determined by
a ring $R$ and an appropriate \'etale sheaf of commutative monoids $M$
on $\Spec(R)$.  We now give some details about the data needed on  
a map $f\co A \to R$, classifying a Weierstrass curve over $\Spec(R)$, 
to get a map $\Spec(R,M) \to \Mlog$, and when this map is log-\'etale.

The map $f$ determines a generalized elliptic curve when the ideal
$(f(c_4),f(\Delta))$ is the unit ideal of $R$.  As in the previous
section, a lift of this to a map of log schemes is a lift of
$\Delta/c_4^3$ to a section of $M$ over $\Spec(c_4^{-1} R)$; the
fact that $M$ is a logarithmic structure implies that this is
equivalent to a lift of $\Delta$ to a section $\tilde \Delta$ of $M$
over $\Spec(R)$.

We then have a composite pullback diagram
\[
\xymatrix{
\Spec(R \otimes_A \Gamma, (\eta_L)^* M) \ar[r] \ar[d] &
R_{log} \ar[r] \ar[d] & U_{log} \ar[d] \\
\Spec(R,M) \ar[r] & U_{log} \ar[r] & \Mlog.
}
\]
We find that the composite of the lower maps is log-\'etale if and
only if the resulting map
\[
(A,\langle\Delta\rangle) \to (R \otimes_A \Gamma, (\eta_L)^* M),
\]
induced by the right unit of $\Gamma$, is log-\'etale.  Here
$(\eta_L)^* M$ is the pullback logarithmic structure.

In particular, an object $\Spec(R,M)$ can be log-\'etale over $\Mlog$
only if $R$ is Noetherian.  In this case, the Artin-Rees lemma
gives us the following.
\begin{prop}
\label{prop:artinrees}
For a log ring $\Spec(R,M)$ log-\'etale over $\Mlog$ and a
finitely-generated $R$-module $N$, the diagram
\[
\xymatrix{
N \ar[r] \ar[d] & \comp{N}_\Delta \ar[d] \\
\Delta^{-1} N \ar[r] & \Delta^{-1} \comp{N}_\Delta
}
\]
is both cartesian and cocartesian in $R$-modules.
\end{prop}

In particular, taking $N = R$ tells us that an object of the
log-\'etale site has its underlying scheme completely determined by an
object $\Delta^{-1} X \to \mell$ \'etale over the moduli of smooth
elliptic curves, an object $(\comp{X}_\Delta,M_X)$ over the completion
at the cusp, and a patching map $\Delta^{-1} \comp{X}_\Delta \to
\Delta^{-1} X$ over $\mell$.

\subsection{The Tate curve}
\label{sec:tate-curve}

For references for the following material, we refer the reader to
\cite[VII]{deligne-rapoport} or \cite[\S2.3]{ando-ellipticpowerops}.

The Tate curve $T$ is a generalized elliptic curve over $\mb Z\pow{q}$
defined by the formula
\[
y^2 + xy = x^3 + a_4(q)x + a_6(q),
\]
where
\begin{alignat*}{2}
a_4(q) &= -5 \sum_{n=1}^{\infty} \frac{n^3 q^n}{1 - q^n}, &\hspace{2pc}
a_6(q) &= -\frac{1}{12} \sum_{n=1}^{\infty} \frac{(5n^3 + 7n^5) q^n}{1 - q^n}.
\end{alignat*}
The $j$-invariant of this curve is
\begin{equation}
  \label{eq:jinvariant}
j(q) = q^{-1} + 744 + 196884q + \ldots
\end{equation}
At $q=0$, the Tate curve is a curve of genus one with a nodal
singularity, whose smooth locus is the group scheme $\mb G_m$.

The Tate curve possesses a chosen isomorphism of formal groups
$\widehat T \cong \widehat{\mb G}_m$ over $\mb Z\pow{q}$ and a
canonical nowhere-vanishing invariant differential.  In addition,
there are compatible diagrams of group schemes
\[\xymatrix{
\mu_n \ar[r] \ar[d] &
T[n] \ar[d] \\
\widehat{\mb G}_m \ar[r] &
T
}\]
as $n$ varies.

For any $n \in \mb N$, define $\psi^n(q) = q^n$.  This map has a lift
$\psi^n_T\co T \to T$ making the following diagram commute:
\[\xymatrix{
T \ar[r]^{\psi^n_T} \ar[d] & T \ar[d] \\
\Spec(\mb Z\pow{q}) \ar[r]_{\psi^n} &
\Spec(\mb Z\pow{q})
}\]
The resulting map $T \to (\psi^n)^* T$ is an isogeny with $\mu_n$
mapping isomorphically to the kernel.

\subsection{The Tate moduli}
\label{sec:tate-moduli}

The Tate curve is classified by a map $\Spec(\mb Z\pow{q}) \to \Mell$.
It also has a $\mb Z/2$-action by negation.

Equation~(\ref{eq:jinvariant}) allows us to express $j^{-1}$ and $q$
as power series in each other, so the coefficient $q$ is uniquely
determined by $j$.  Both $j$ and $(j-1728)$ are the inverses of
elements in $\mb Z\pow{q}$, and as in Lemma~\ref{lem:cuspcoords} this
means that there are no isomorphisms between generalized elliptic
curves parametrized by the Tate curve other than the identity and
negation.  In particular, the automorphism scheme $\underbar{Aut}(T)$
is the constant group scheme $\mb Z/2$.

This implies that we have an identification of fiber products
\[
\Spf(\mb Z\pow{q}) \times_{\Mell} \Spf(\mb Z\pow{q}) \cong \Spf(\mb
Z\pow{q}) \times \mb Z/2.
\]
In particular, the substack of $\Mell$ parametrizing generalized
elliptic curves locally isomorphic to the Tate curve is isomorphic to
the quotient stack
\[
\mtate = [\Spf(\mb Z\pow{q}) /\!/ \mb Z/2 ].
\]
The following shows that this identification is compatible with the
logarithmic structure.

\begin{prop}
Log maps to $\mtate$ are the same as log maps to $\Spf(\mb
Z\pow{q},\langle q\rangle)$, together with a choice of principal
$\mb Z/2$-torsor.
\end{prop}

\begin{proof}
Our identification of $\mtate$ has already shown that maps $X  \to
\mtate$ are equivalent to maps $j^{-1}\co X \to \Spf(\mb Z\pow{q})$,
together with a principal $\mb Z/2$-torsor $Y = X \times_{\mtate}
\Spf(\mb Z\pow{q})$.

As the $j$-invariant of the Tate curve has the expression from
Equation~(\ref{eq:jinvariant}), it is a unit times $q$.  Therefore,
the logarithmic structure $\langle j^{-1}\rangle$ on $\Spf(\mb
Z\pow{q})$ is the isomorphic to the logarithmic structure $\langle q
\rangle$.
\end{proof}

If $(X,M_X)$ is a log scheme with a map $X \to \mtate$, an extension to
a log map is equivalent to a choice of lift of $q$ to an element
$\tilde q \in M_X(X)$.

Let $\multmod = [\Spec(\mb Z) /\!/ \mb Z/2]$ be the substack of
$\mtate$ defined by $q=0$.

\begin{defn}
For a log scheme $(X,M_X)$ log-\'etale over $\Mlog$ carrying the
generalized elliptic curve ${\cal E}$, the cusp divisor lifts to a
smooth divisor $X^c \subset X$ with a map $X^c \to \multmod$.

The associated form of the multiplicative group scheme is the smooth
locus of the restriction ${\cal E}|_{X^c}$, classified by the
resulting map $X^c \to \multmod$.
\end{defn}

We note that the cusp subscheme is functorial, and that there is a
natural diagram
\[
\xymatrix{
X^c \ar[r] \ar[d] & X \ar[d] \\
\multmod \ar[r] & \Mlog.
}
\]
If $(X,M_X)$ is log-\'etale over $\Mlog$, then $X^c$ is \'etale over
over $\multmod$.

\subsection{Modular curves}
\label{sec:modular}

In this section we will describe why the modular curves, for various
forms of level structure, give a natural tower of log-\'etale maps to
$\Mlog$.

In the following, let $\widehat{\mb Z}$ be the profinite completion of
the integers, let $G = \GL_2(\widehat{\mb Z})$, and for $N \geq 1$ let
$p_N$ be the surjection $G \twoheadrightarrow \GL_2(\mb Z/N)$.

\begin{defn}
\label{def:levelstruct}
The category $\cal{L}$ of level structures is defined as follows.  The
objects are pairs $(N,\Gamma)$ of a positive integer $N$ and a
subgroup $\Gamma < \GL_2(\mb Z/N)$, with
\[
\Hom_{\cal{L}}((N,\Gamma), (N',\Gamma')) = 
\begin{cases}
\Hom_G(G/(p_N^{-1} \Gamma), G/(p_{N'}^{-1} \Gamma')) &\text{if }N'
| N\\
\emptyset &\text{otherwise,}
\end{cases}
\]
\end{defn}
In particular, morphisms in $\cal{L}$ are generated by morphisms
of the following three types:
\begin{enumerate}
\item inclusions $\Gamma < \Gamma'$ for subgroups of $\GL_2(\mb Z/N)$,
\item conjugation maps $\Gamma \to g\Gamma g^{-1}$ for elements
  $g \in \GL_2(\mb Z/N) / \Gamma$, and
\item changes-of-level $(NM,p^{-1} \Gamma) \to (N,\Gamma)$, where
  $p\co \GL_2(\mb Z/NM) \to \GL_2(\mb Z/N)$ is the projection.
\end{enumerate}

\begin{prop}
\label{prop:modulartower}
For any pair $(N,\Gamma)$ in $\cal{L}$, there is a Deligne-Mumford
stack denoted $\MMM(\Gamma)$, parametrizing elliptic curves with level
$\Gamma$ structure over $\mb Z[1/N]$.  Moreover, this is functorial in
the sense that there is a weak 2-functor from $\cal{L}$ to the
2-category of Deligne-Mumford stacks.
\end{prop}

\begin{proof}
The existence of a functorial family of smooth Deligne-Mumford stacks
$\MMM(\Gamma)$ in $\Sch$ parametrizing generalized elliptic curves
with level $\Gamma$ structure, away from the primes dividing the
level, is from \cite[{\S}IV.3]{deligne-rapoport}.
\end{proof}

These have natural open substacks ${\cal M}(\Gamma)$ parametrizing
smooth elliptic curves, and there is an associated direct image log
structure (\ref{def:directimage}).  Specifically, the submonoid of
${\cal O}_X$ of functions invertible away from the cusps defines a
logarithmic structure on $\MMM(\Gamma)$, natural in $(N,\Gamma)$.

\begin{prop}
\label{prop:cechnerve}
The functor of Proposition~\ref{prop:modulartower} extends to log
stacks.  For a fixed $N$, any map $(N,\Gamma) \to (N,\Gamma')$ in
${\cal L}$ induces a log-\'etale covering map $\MMM(\Gamma') \to
\MMM(\Gamma)$.  If $K \lhd \Gamma < \GL_2(\mb Z/N)$, then the \v Cech
nerve of the associated covering map in log stacks is the simplicial
bar construction for the action of $\Gamma/K$ on $\MMM(K)$.
\end{prop}

\begin{proof}
The objects $\MMM(\Gamma)$ are the normalizations of $\Mell[1/N]$ in
$\mell(\Gamma)[1/N]$ \cite[IV.3.10]{deligne-rapoport} and have tame
ramification over the cusps.  Proposition~\ref{prop:coveringequiv}
establishes that this data is equivalent to a system of Kummer
log-\'etale covering maps of $\Mlog$.  Moreover,
Corollary~\ref{cor:smcd} and Proposition~\ref{prop:kummer} together
show that any such cover has logarithmic structure determined by a
smooth divisor.

In particular, $\MMM(\Gamma')$ is also the normalization of
$\MMM(\Gamma)$ in $\mell(\Gamma')[1/N]$, and the map $\MMM(\Gamma')
\to \MMM(\Gamma)$ is also a log-\'etale covering map.

We now consider the inclusion of a normal subgroup.  The action of
$\Gamma/K$ on $\MMM(K)$ over $\MMM(\Gamma)$ gives rise to a map from
the simplicial bar construction $(\Gamma/K)^{\bullet} \times Y$ to the
\v Cech nerve; to show that it is an isomorphism, it suffices by
induction to show in degree $2$ that the shearing map
\[
\Gamma/K \times \MMM(K) \to \MMM(K) \times_{\MMM(\Gamma)} \MMM(K)
\]
is an isomorphism.  However, this is a log-\'etale map covering map
which is an isomorphism over $\mell$, and so this is true by
Proposition~\ref{prop:coveringequiv}.
\end{proof}

On any \'etale open $U \to \MMM(\Gamma)$, the logarithmic structure is
defined by the cusp divisor $U^c$. As in the previous section, the
logarithmic structure determines a natural cusp substack
$\MMM(\Gamma)^c$ which is \'etale over $\multmod$.

\subsection{Grothendieck sites}
\label{sec:wlog}

For convenience, we consider the following category of elliptic
curves.  Recall that $(U_{log},R_{log})$
(\S\ref{sec:weierstrass-curves}) forms a groupoid in log schemes,
parametrizing generalized elliptic curves in Weierstrass form with
compatible logarithmic structures.

\begin{defn}
\label{def:weier}
The full subcategory $\Wlog \subset \ALS/\Mlog$ is defined as follows. 
The objects of $\Wlog$ are fine and saturated log schemes of the form
$\Spec(R,M)$, equipped with a generalized elliptic curve ${\cal E}$ in
Weierstrass form and a lift of the elliptic discriminant to $\tilde
\Delta \in M(\Spec(R))$, such that the associated map to $\Mlog$ is
log-\'etale (\S\ref{sec:log-etale-objects}).  Maps are maps
over $\Mlog$. 
\end{defn}

Equivalently, the objects of $\Wlog$ are affine log schemes
$\Spec(R,M) \to U_{log}$ such that the composite $\Spec(R,M) \to
U_{log} \to \Mlog$ is log-\'etale.

The category $\Wlog$, while it is not closed under limits in the small
\'etale site of $\Mlog$, still inherits a Grothendieck topology.
\begin{prop}
\label{prop:siteequivalence}
The inclusion from $\Wlog$ to the small log-\'etale
site of $\Mlog$ is an equivalence of Grothendieck sites.
\end{prop}
\begin{proof}
  It suffices, by \cite[C.2.2.3]{johnstone-elephant-v2} to show that
  any object Kummer log-\'etale over $\Mlog$ has an \'etale cover by
  objects isomorphic to those from $\Wlog$.  However, this merely
  expresses the fact that a log scheme can be covered by affine
  charts, and that elliptic curves on affine schemes are locally
  isomorphic to elliptic curves in Weierstrass form.
\end{proof}

As our goal is to construct a presheaf ${\cal O}$ on the small
log-\'etale site of $\Mlog$ satisfying homotopy descent, it will
ultimately be the case that we can equivalently carry out a
construction on $\Wlog$ (see \S \ref{sec:sheafification}).

\section{Homotopy theory}
\label{sec:homotopy-theory}

In this section we will discuss several important tools. We will
assume that the reader is familiar with some more fundamental topics
in homotopy theory: model categories, homotopy limits, smash products,
and the derived smash product. (Appendix A of \cite{lurie-htt}
provides a very convenient reference for much of the material we will
be using on combinatorial model categories.)

Most of the material in this chapter is the work of other authors, but
is compiled here for convenience.

\subsection{Elliptic cohomology theories}
\label{sec:lifting}

We being by recalling the following (see
\cite{ando-hopkins-strickland-witten,lurie-survey,mark-construct}).

\begin{defn}
A homotopy-commutative ring spectrum $E$ is weakly even-periodic if
$\pi_n E$ is zero for $n$ odd, and if the tensor product $\pi_p E
\otimes_{\pi_0 E} \pi_q E \to \pi_{p+q}E$ is an isomorphism for $p$,
$q$ even. It is even-periodic if $\pi_2 E \cong \pi_0 E$, or
equivalently if $\pi_2 E$ contains a unit.
\end{defn}

\begin{prop}
If $E$ is a weakly even-periodic spectrum, then
$\Spf(E^0(\mb{CP}^\infty))$ is a smooth, 1-dimensional formal group
$\mb G_E$ over the ring $\pi_0 E$. There is a natural identification of
$\pi_{2t} E$ with the tensor power $\omega^{\otimes t}$ of sheaf of
invariant differentials of $\mb G_E$.
\end{prop}

\begin{defn}
An elliptic spectrum consists of a weakly even-periodic spectrum $E$,
an elliptic curve ${\cal E}$ over $\pi_0 E$, and an isomorphism
$\alpha\co \mb G_E \to \widehat{\cal E}$ between the formal group of
the complex orientable theory and the formal group of ${\cal E}$ over
$\pi_0 E$.  We say that this elliptic spectrum realizes the elliptic
curve ${\cal E}$.

A map of elliptic spectra is a multiplicative map $E \to E'$ together
with a compatible isomorphism ${\cal E}' \to {\cal E} \otimes_{\pi_0
  E} \pi_0 E'$ of elliptic curves which respects the isomorphisms of
formal groups.  We will say that a diagram of elliptic spectra
realizes the corresponding diagram of elliptic curves.
\end{defn}

\begin{rmk}
Consider the case where ${\cal E}$ is given the structure of a
Weierstrass curve.  It carries a coordinate $-x/y$ near the unit which
trivializes the sheaf $\omega$ of invariant differentials, and hence a
choice of Weierstrass representation determines a canonical
identification $\pi_* E \cong \pi_0 E[u^{\pm 1}]$.  (This also gives
its formal group a standard lift to a formal group law classified by a
map from the Lazard ring to $\pi_0 E$, and $E$ has a corresponding
standard orientation.)
\end{rmk}

We will find the following result convenient in showing that many
objects defined by pullback naturally remain elliptic
spectra (compare \cite[3.9]{level3}).

\begin{lem}
\label{lem:orientationpullback}
Suppose that we have a homotopy pullback diagram
\[
\xymatrix{
R \ar[r] \ar[d] & S \ar[d] \\
S' \ar[r] & T
}
\]
of maps of homotopy commutative ring spectra, all of which are weakly
even-periodic and complex orientable.  Suppose $\pi_0 R$ carries an
elliptic curve ${\cal E}$, and that the subdiagram $S \to T \leftarrow
S'$ is a diagram of elliptic spectra realizing ${\cal E}$.  Then there
is a unique way to give $R$ the structure of an elliptic spectrum
realizing ${\cal E}$ so that the square commutes.
\end{lem}

\begin{proof}
The diagram of elliptic spectra gives us a commutative diagram of
formal groups
\begin{equation}
  \label{eq:pullbackrows}
\xymatrix{
\mb G_R \otimes \pi_0 S \ar[d]^\sim &
\mb G_R \otimes \pi_0 T \ar[l] \ar[r] \ar[d]^\sim &
\mb G_R \otimes \pi_0 S' \ar[d]^\sim &\\
\widehat{\cal E} \otimes \pi_0 S &
\widehat{\cal E} \otimes \pi_0 T \ar[r] \ar[l] &
\widehat{\cal E} \otimes \pi_0 S'. 
}
\end{equation}
Here the tensor products are taken over $\pi_0 R$.  The hypotheses
imply that, for any $n$, applying $\pi_n$ to the homotopy pullback
diagram gives a bicartesian square.  Therefore, taking pullbacks along
rows of the coordinate rings from Equation~(\ref{eq:pullbackrows})
gives us the formal groups $\mb G_R$ and $\widehat{\cal E}$, and so
there is a unique isomorphism $\mb G_R \to \widehat{\cal E}$
compatible with the given isomorphisms.  In addition, the resulting
pullback diagram for $\omega^{\otimes n}$ shows that this isomorphism
gives $R$ the structure of an elliptic spectrum.
\end{proof}

\begin{rmk}
In this paper, we will often be working with strictly
even-periodic ring spectra rather than the weak version.  This
simplifies discussion of localization and completion with respect to
elements in the ring of modular forms.
\end{rmk}

\subsection{Rectification of diagrams}

In this section we will describe how to convert homotopy coherent
diagrams into strict diagrams. Throughout it we will fix a simplicial
model category ${\cal M}$ which is combinatorial (locally presentable
and cofibrantly generated).

\begin{defn}
  For a simplicial category $I$, let $\pi_0 I$ be the category with
  the same objects and with $\Hom_{\pi_0 I}(i,i') = \pi_0
  \Hom_I(i,i')$.
\end{defn}

\begin{defn}
  For a small simplicial category $I$ and a simplicial category ${\cal
    C}$, let ${\cal C}^I$ denote the category of simplicial functors
  $I \to {\cal C}$ and natural transformations.
\end{defn}

\begin{thm}[{\cite[A.3.3.2]{lurie-htt}}]
  If $I$ is a small simplicial category, there is a projective model
  structure on ${\cal M}^I$, where a natural transformation $f
  \to g$ is a weak equivalence or fibration if and only if $f(i) \to
  g(i)$ has the corresponding property for each object $i \in I$.
\end{thm}

\begin{defn}
  A map $f\co I \to J$ of simplicial categories is a
  Dwyer-Kan equivalence if
  \begin{itemize}
  \item $\pi_0 f\co \pi_0 I \to \pi_0 J$ is essentially surjective,
    and
  \item for any $i, i' \in I$, the function $\Map_I(i,i') \to \Map_J(fi,
    fi')$ is a weak equivalence of simplicial sets.
  \end{itemize}
\end{defn}

\begin{thm}[{\cite[A.3.3.8, A.3.3.9]{lurie-htt}}]
\label{thm:rectification}
  For a combinatorial simplicial model category ${\cal M}$ and a
  simplicial functor $f\co I \to J$ between small simplicial
  categories, the restriction functor $f^*\co {\cal M}^J \to
  {\cal M}^I$ is a right Quillen functor between the projective
  model structures. The left adjoint, denoted by $f_!$, is given by
  left Kan extension.

  If $f$ is a Dwyer-Kan equivalence, then the Quillen adjunction
  $(f_!, f^*)$ is a Quillen equivalence.
\end{thm}

\begin{defn}
  A simplicial category is homotopically discrete if the natural
  functor $I \to \pi_0 I$ is a Dwyer-Kan equivalence: all the
  simplicial sets $\Hom_I(i,i')$ are weakly equivalent to discrete
  sets.
\end{defn}

\begin{cor}
\label{cor:discreterectification}
  Suppose that $I$ is a small simplicial category which is
  homotopically discrete. Then any simplicial functor $g\co I \to {\cal M}$
  is naturally weakly equivalent to a functor $\pi_0 I \to {\cal M}$,
  and any two such are naturally weakly equivalent.
\end{cor}

\begin{proof}
This follows from the fact that the map ${\cal M}^{\pi_0 I}
\to {\cal M}^I$ induces an isomorphism on homotopy
categories. If we wish, we can be slightly more explicit: writing
$\pi$ for the projection $I \to \pi_0 I$, we have a diagram of
natural weak equivalences
\[
g \leftarrow g_{cof} \to \pi^* \pi_! g_{cof},
\]
where $g_{cof}$ is a cofibrant replacement of $g$ in ${\cal M}^I$.
\end{proof}

\begin{rmk}
Our references are to recent work, but this result and those like it
have a long history. Some particularly relevant ones include the
following, non-exhaustive catalogue. Rectification for coherent
diagrams appeared in work of Vogt and Dwyer-Kan-Smith; Cordier-Porter
gave a proof for homotopy classes of diagrams in categories possessing
sufficient limits; Segal developed the specific mechanism of
using $\pi^* \pi_!$ in diagrams of spaces; Devinatz-Hopkins gave a
proof in the case of $E_\infty$ ring spectra; and Elmendorf-Mandell
developed this particular adjunction in the multicategory of symmetric
spectra.
\end{rmk}


\subsection{Rings and ring maps}

In order to make our constructions concrete, we will need to fix a
particular model for the theory of $E_\infty$ ring spectra. One
convenient place to do so is the category of symmetric spectra
\cite{hovey-shipley-smith-symmetric}.

We recall that the category of symmetric spectra has a symmetric
monoidal structure $\wedge$ whose unit is the sphere spectrum $\mb S$,
as well as a symmetric monoidal functor $\Sigma^\infty$ from based
simplicial sets to symmetric spectra.

For a symmetric spectrum $X$, we take our definition of the $n$'th
homotopy group $\pi_n(X)$ to be the abelian group $[S^n,X]$ of maps in
the homotopy category of symmetric spectra. In particular,
equivalences of symmetric spectra are the same as $\pi_*$-isomorphisms
with this definition.

\begin{defn}
The category of commutative (symmetric) ring spectra is the category
of commutative monoids under $\wedge$.

If $R$ is a commutative ring spectrum, the category of commutative
$R$-algebras is the category of commutative ring spectra under $R$,
and the category of $R$-modules is the category of symmetric spectra
with a (left) action of $R$, with its induced symmetric monoidal
structure $\smsh{R}$.
\end{defn}

The homotopy groups $\pi_* R$ of a commutative ring spectrum form a
graded-commutative ring.

\begin{defn}
For $M$ a commutative monoid, the monoid algebra $\mb
S[M]$ is the commutative ring spectrum $\Sigma^\infty M_+$. Similarly,
for $M'$ a based monoid (i.e. a monoid with zero element), the based
monoid algebra $\mb S[M']$ is the commutative ring spectrum
$\Sigma^\infty M'$.
\end{defn}
\begin{rmk}
The ring $\pi_*(\mb S[M])$ is the monoid algebra $(\pi_* \mb S)[M]$.
\end{rmk}

\begin{defn}
Let $\mb P$ be the monad taking a symmetric spectrum $X$ to the free
commutative ring spectrum on $X$:
\[
\mb P(X) = \bigvee_{n \geq 0} X^{\wedge n}/\Sigma_n.
\]
The category of commutative ring spectra is equivalent to the category
of $\mb P$-algebras.

Similarly, if $R$ is a commutative ring spectrum, we write $\mb P_R(X)$
for the free commutative $R$-algebra on an $R$-module $X$.
\end{defn}

\begin{rmk}
\label{rmk:rationalpolynomials}
Upon rationalization, the ring $\pi_* \mb P_R(X)$ is the free $\pi_*
R_{\mb Q}$-algebra on the module $\pi_* X_{\mb Q}$ (assuming that $X$
is cofibrant as an $R$-module).
\end{rmk}

\begin{defn}
For commutative ring spectra $S$ and $T$, let $\Map_{comm}(S,T)
\subset \Map(S,T)$ be the simplicial set of commutative ring maps $S
\to T$. If both $S$ and $T$ are commutative $R$-algebras for some
fixed $R$, let $\Map_{comm/R}(S,T) \subset \Map_{comm}(S,T)$ be the
simplicial set of commutative $R$-algebra maps.
\end{defn}

There is a positive flat stable model structure on symmetric spectra
(called the positive $S$-model structure in \cite{shipley-convenient})
so that the following result holds.

\begin{thm}
\label{thm:simplicialmodel}
There exists a combinatorial simplicial model structure on commutative
ring spectra such that the forgetful functor to symmetric spectra is a
right Quillen functor, with adjoint $\mb P$. (The same result holds
for commutative $R$-algebras, $R$-modules, and $\mb P_R$.)
\end{thm}

\begin{rmk}
The existence of the model structure is (3.2) in the referenced paper,
and the statement that these model categories are locally presentable
is from the second page. The cited reference does not directly show
that the categories of symmetric spectra or commutative ring spectra
are simplicial model categories. However, as described in
\cite{white-commutativemonoids}, it is straightforward to verify that
these model categories are simplicial: in all these categories,
fibrations and cotensors with a simplicial set are created in
symmetric spectra, and so the pullback formulation of axiom SM7 is
inherited from a levelwise verification in symmetric spectra.
\end{rmk}

In combination with Corollary~\ref{cor:discreterectification}, this
combinatorial model structure allows us to refine diagrams of
commutative ring spectra with simple mapping spaces to strict
diagrams. To proceed further in this direction, we will require some
results from obstruction theory.

There have been several versions of obstruction theory developed for
the classification of objects and maps in commutative ring
spectra. Most of these produce obstructions in groups derived from
the cotangent complex. As a result, when the cotangent complex
vanishes we obtain strong uniqueness results.

\begin{prop}
\label{prop:etalerings}
Suppose $R$ is a commutative ring spectrum and $\pi_* R \to A_*$ is a
map of evenly-graded commutative rings which is \'etale. Then there
exists an $R$-algebra $S$ equipped with an isomorphism $\pi_* S \cong
A_*$ of $\pi_* R$-algebras, and for any such $S$ the natural
transformation
\[
\pi_* \co \Map_{comm/R}(S, -) \to \Hom_{gr.comm/\pi_* R}(A_*,-),
\]
from the mapping space to the discrete set, is a natural weak
equivalence.
\end{prop}

\begin{proof}
Our proof roughly follows
\cite[\S2.2]{baker-richter-algebraicgalois}. The Goerss-Hopkins
obstruction theory in the category of $R$-modules
\cite{goerss-hopkins-summary}, using the auxiliary homology theory $E 
= R$, gives obstructions to existence of such a commutative
$R$-algebra $S$ in the Andr\'e-Quillen cohomology groups
\[
Der^{s+2}_{\pi_* R}(A_*, \Omega^s A_*),
\]
and for any such $S$ there is a Bousfield-Kan spectral sequence
converging to $\pi_{t-s} \Map_{comm/R}(S,T)$ with
\[
E_2^{s,t} = \begin{cases}
  Hom_{gr.comm/\pi_* R}(A_*,\pi_* T) & (s,t) = (0,0)\\
  Der^s_{\pi_* R}(A_*, \Omega^t \pi_* T) & s > 0.
\end{cases}
\]
For \'etale $\pi_* R$-algebras the Andr\'e-Quillen cohomology
groups $Der^s_{\pi_* R}(A_*,-)$ vanish identically for $s > 0$, for
example by identifying them with $\Gamma$-cohomology
\cite{basterra-richter-gammacohomology} and applying vanishing results
of Robinson-Whitehouse \cite{robinson-whitehouse-gammacohomology}.
\end{proof}

This result is sometimes referred to as ``topological invariance of
the \'etale site.'' (A related version with a more restricted notion
of \'etale maps appears in
\cite[7.5.0.6,7.5.4.6]{lurie-higheralgebra}.) Combining it with
Corollary~\ref{cor:discreterectification}, we find the following.
\begin{cor}
\label{cor:etalespheres}
There exists a functor from \'etale extensions of $\mb Z$ to
commutative ring spectra, sending a ring $A$ to a ring spectrum $\mb S
\otimes A$ such that $\pi_* (\mb S \otimes A) \cong \pi_* \mb S
\otimes A$.
\end{cor}

We recall that the Eilenberg-Mac Lane functor $H$, from graded abelian
groups to modules over $H\mb Z$, extends to a functor
from graded-commutative $\mb Z$-algebras to commutative $H\mb
Z$-algebras (corresponding to the underlying differential graded
algebra with trivial differential).

\begin{defn}
A commutative ring spectrum $R$ is formal if it is equivalent, in the
homotopy category of commutative ring spectra, to $H\pi_* R$.

More generally, a functor from a category $I$ to the category of
commutative ring spectra is formal if it is weakly equivalent to
a functor factoring through $H$.
\end{defn}

\begin{prop}
\label{prop:etaleformality}
Suppose $R$ is an evenly-graded commutative ring spectrum which is
formal, and let $\Et/R$ be the category of evenly-graded
commutative $R$-algebras $T$ such that $\pi_* R \to \pi_* T$ is 
\'etale. Then, for any small category $I$ with a functor $f\co I \to
\Et/R$, $f$ is equivalent to the functor $H \pi_* f$.
\end{prop}

\begin{proof}
Since $H\pi_* R$ is equivalent to $R$, without loss of generality we
may replace both $f\co I \to \Et/R$ and $H\pi_*f\co \pi_0 I \to \Et/R$
with equivalent functors taking values in cofibrant-fibrant
commutative $R$-algebras.

Let $J$ (resp. $J'$) be the full subcategory of $\Et/R$ whose object
set is the image of $f$ (resp. the union of the images of $f$ and
$H\pi_* f$). The composite map
\[
\xymatrix{
\pi_0 J \ar[r]^{H} &
J' \ar@{>>}[r] &
\pi_0 J'
}
\]
is an equivalence of categories, and by
Proposition~\ref{prop:etalerings} the right-hand projection is a
Dwyer-Kan equivalence. Therefore, $H$ is also a Dwyer-Kan equivalence. 

Applying Theorem~\ref{cor:discreterectification} to $H\co \pi_0 J
\to J'$, we find that the identity functor $J' \to \Et/R$ is
equivalent to a functor factoring through $H$. The result follows by
precomposition with $f$.
\end{proof}

\begin{rmk}
We might think of this as a consequence of the fact that the functor
$\pi_* \co \Et/R \to \Et/\pi_* R$ is a Dwyer-Kan equivalence (although
the categories in question are not small categories). In particular, we
could apply the previous result to the inclusion of a full subcategory
of $\Et/R$ which maps to a skeleton of $\Et/\pi_* R$, and then use a
homotopical Kan extension to construct a homotopy inverse functor to
$\pi_*$ on all of $\Et/R$.
\end{rmk}

\subsection{Homotopy Sheafification}
\label{sec:sheafification}

In the following we describe how to construct homotopical
sheafifications that we will require. Most of the methods of this
section are taken directly, or with slight modification, from
\cite[\S2]{mark-construct}.

In this chapter, let ${\cal C}$ be a small Grothendieck site. Recall
that we have a positive flat stable model structure on symmetric
spectra from Theorem~\ref{thm:simplicialmodel}.
\begin{defn}
A map $X \to Y$ of presheaves of symmetric spectra on ${\cal C}$ is
\begin{itemize}
\item a level cofibration if, for any $U \in {\cal C}$, the map $X(U) \to
  Y(U)$ is a cofibration,
\item a level equivalence if, for any $U \in {\cal C}$, the map $X(U)
  \to Y(U)$ is an equivalence, and
\item a local equivalence if the map of presheaves of homotopy groups
  $\pi_* X(U) \to \pi_* Y(U)$ induces an isomorphism on the associated
  sheaves $\underline{\pi}_* X \to \underline{\pi}_* Y$.
\end{itemize}
\end{defn}

\begin{prop}[{\cite[A.3.3.2, A.3.3.6]{lurie-htt}}]
The level cofibrations and level equivalences define combinatorial
simplicial model structures (the injective model structures) on the
categories of presheaves of symmetric spectra and commutative ring
spectra. Under these, the forgetful functor from presheaves of
commutative ring spectra to presheaves of symmetric spectra is a right
Quillen functor whose left adjoint is $\mb P$.
\end{prop}

As these model structures are combinatorial, we may then localize with
respect to the local equivalences.
\begin{prop}[{\cite[A.3.7.3]{lurie-htt}}]
The level cofibrations and local equivalences define combinatorial
simplicial model structures on the categories of presheaves of
symmetric spectra and commutative ring spectra. Under these, the
forgetful functor from presheaves of commutative ring spectra to
presheaves of symmetric spectra is a right Quillen functor whose left
adjoint is $\mb P$.
\end{prop}

\begin{rmk}
For a commutative ring spectrum $R$, similar results hold in the
categories of presheaves of $R$-modules and $R$-algebras.
\end{rmk}

\begin{rmk}
\label{rmk:jardinerelation}
Under the ordinary, non-positive stable model structure on symmetric
spectra where cofibrations are simply monomorphisms, the existence of
this model structure is due to Jardine
\cite{jardine-presheavesspectra}. It is almost certain that the proof
given there, which is axiomatic in flavour, goes through
verbatim. Given that, we could instead recover this adjunction from
the ``pointwise'' symmetric monoidal structure on presheaves by work
of White \cite{white-commutativemonoids}.  The identity functor is the
left adjoint in a Quillen equivalence between the positive model
category of presheaves and Jardine's ordinary model category of
presheaves.
\end{rmk}

\begin{prop}
\label{prop:fibrantdescent}
Suppose ${\cal F}$ is a fibrant presheaf of symmetric spectra on
${\cal C}$. Then ${\cal F}$ satisfies homotopy descent with respect to
hypercovers: for any hypercover $U_\bullet \to X$, the map
\[
{\cal F}(X) \to \holim {\cal F}(U_\bullet)
\]
is an equivalence.

In this situation, there is also a hypercohomology spectral sequence
\[
H^s(X, \underline \pi_t({\cal F}))
\Rightarrow \pi_{t-s} {\cal F}(X)
\]
calculating the values of $\pi_* {\cal F}(X)$.
\end{prop}

\begin{proof}
We may construct a weak equivalence ${\cal F} \to {\cal F}_{Jf}$ to
a Jardine fibrant presheaf of symmetric spectra, which is also fibrant
in our model structure by Remark~\ref{rmk:jardinerelation}. The map
${\cal F} \to {\cal F}_{Jf}$ is an equivalence between fibrant
objects, and hence an equivalence in the injective model structure.

Therefore, ${\cal F} \to {\cal F}_{Jf}$ is a level equivalence. We
therefore recover the homotopy descent property and the stated
spectral sequence from the fact that these hold for the presheaf
${\cal F}_{Jf}$.
\end{proof}

\begin{prop}
Suppose ${\cal F}$ is a presheaf of symmetric spectra on
${\cal C}$ which satisfies homotopy descent with respect to
hypercovers. Then any fibrant replacement ${\cal F} \to {\cal F}_f$ is
also a level equivalence.
\end{prop}

\begin{proof}
This property is independent of the choice of fibrant replacement, so
it suffices to show that it holds for one.  We take a level
equivalence ${\cal F} \to {\cal F}_{If}$ which is a fibrant
replacement in the ordinary injective model structure. In particular,
${\cal F}_{If}$ is injective fibrant in the Jardine model structure
and satisfies homotopy descent for hypercovers by comparison with
${\cal F}$. By work of Dugger-Hollander-Isaksen, ${\cal F}_{If}$ is
Jardine fibrant in the local model structure
\cite{dugger-hollander-isaksen}, and therefore also a fibrant
replacement for ${\cal F}$ in our local model structure.
\end{proof}

\begin{prop}
\label{prop:changeofsite}
Suppose $i\co {\cal C} \to {\cal D}$ is a continuous functor between
Grothendieck sites. Then the restriction functor $i^*$ on 
presheaves of commutative ring spectra is a left Quillen functor under
the local model structure, and is a Quillen equivalence if $i$ is both
fully faithful and an equivalence of sites.
\end{prop}

\begin{proof}
The functor $i^*$ clearly preserves cofibrations, and possesses a
right adjoint $i_*$ (given by right Kan extension). Moreover,
sheafification commutes with $i^*$, and so the restriction of a local
equivalence is a local equivalence.

If $i$ is an equivalence of sites, then the map of sheaves
$\underline{\pi}_* X \to \underline{\pi}_* Y$ is isomorphic to the map
of sheaves $\underline{\pi}_* (i^*X) \to \underline{\pi}_* (i^*Y)$,
and so $i^*$ preserves and reflects weak equivalences. 

If $i$ is also fully faithful, then we have
\[
(i_* X)(i(U)) = \lim_{V \to i(U)} X(V) \cong X(i(U))
\]
and so the map $i^* i_* X \to X$ is always an isomorphism; in
particular, the (derived) counit of the adjunction is always an
equivalence.
\end{proof}

\begin{prop}
\label{prop:sheafifyme}
Suppose ${\cal D}$ is the Kummer log-\'etale site of a log stack
$(X,M_X)$, and $i\co {\cal C} \to {\cal D}$ is a fully faithful
equivalence of sites where the objects of ${\cal C}$ are
affine. Suppose ${\cal F}$ is a presheaf of symmetric spectra on
${\cal C}$ such that the presheaves $\pi_k {\cal F}$ are the
restrictions of quasicoherent sheaves to ${\cal C}$. Then any fibrant
replacement ${\cal F} \to {\cal F}_f$ is a level equivalence.
\end{prop}

\begin{proof}
Let ${\cal F} \to {\cal F}_f$ be a fibrant replacement.  The sheaf
$\underline \pi_k ({\cal F}_f)$ is the quasicoherent sheaf $\pi_k
{\cal F}$ on ${\cal C}$. By Corollary~\ref{cor:quasicoherentcohomology},
if $(U,M_U)$ an object of ${\cal C}$, the hypercohomology spectral
sequence of Proposition~\ref{prop:fibrantdescent} calculating the
homotopy groups of ${\cal F}_f(U,M_U)$ degenerates to an isomorphism
\[
\pi_k {\cal F}(U,M_U) \overto^\sim \pi_k {\cal F}_f(U,M_U),
\]
as desired.
\end{proof}

\begin{cor}
\label{cor:cohomologysseq}
Suppose ${\cal D}$ is the Kummer log-\'etale site of a log stack
$(X,M_X)$, and $i\co {\cal C} \to {\cal D}$ is a fully faithful
equivalence of sites where the objects of ${\cal C}$ are
affine. Suppose ${\cal F}$ is a presheaf of commutative ring spectra
on ${\cal D}$ such that the presheaves $\pi_k {\cal F}$ are the
restrictions of quasicoherent sheaves to ${\cal C}$. Then there exists
a fibrant presheaf ${\cal G}$ of commutative ring spectra on ${\cal
  D}$ whose restriction $i^* {\cal G}$ to ${\cal C}$ is equivalent to
${\cal F}$.

For a log stack $(Y,M_Y)$ log-\'etale over $(X,M_X)$, there is a
spectral sequence
\[
H^s_{Zar}((Y,M_Y), \pi_t {\cal F}) \Rightarrow \pi_{t-s} {\cal
  G}(Y,M_Y).
\]
\end{cor}

\begin{proof}
By Proposition~\ref{prop:siteequivalence}, the map $i^*$ is the left
adjoint in a Quillen equivalence. Let ${\cal F} \to {\cal F}_f$ be a
fibrant replacement, which is also a level equivalence. Letting ${\cal
  G} = i_* {\cal F}_f$, the adjoint map $i^* {\cal G} \to {\cal F}_f$
is an equivalence.

Corollary~\ref{cor:quasicoherentcohomology} then allows us to
interpret the resulting hypercohomology spectral sequence as Zariski
cohomology.
\end{proof}

By applying Proposition~\ref{prop:cechnerve}, we can get a homotopy
fixed-point description of the values of such a presheaf on modular
curves.
\begin{prop}
\label{prop:galoissseq}
Let ${\cal G}$ be a fibrant presheaf of symmetric spectra on $\Mell$.
If $K \lhd \Gamma < \GL_2(\mb Z/N)$, then the natural map
\[
{\cal G}(\MMM(\Gamma)) \to {\cal G}(\MMM(K))^{h\Gamma/K}
\]
is an equivalence.  In particular, there is a group cohomology
spectral sequence
\[
H^s(\Gamma/K, \pi_t {\cal G}(\MMM(K)) \Rightarrow \pi_{t-s} {\cal
  G}(\MMM(\Gamma)).
\]
\end{prop}

\begin{rmk}
We could, if desired, take the further step of turning presheaves
into sheaves, using a further Quillen equivalence to a Joyal model
structure on sheaves. This makes comparisons between equivalent
Grothendieck sites more direct. However, for the purposes of this
paper the homotopy sheaf property is sufficient.
\end{rmk}

\subsection{$K$-theory and $Tmf$}

We will require some details about the construction of forms of
$K$-theory from \cite[Appendix~A]{tmforientation}. Let $K$ be the
periodic complex $K$-theory spectrum, equipped with its action of
the cyclic group $C_2$ by conjugation.

We recall that, if $X$ is a spectrum with $C_2$-action, the homotopy
fixed-point spectrum $X^{hC_2}$ is the homotopy limit of the
corresponding diagram.

\begin{prop}[{\cite[Appendix~A]{tmforientation}}]
There is a fibrant presheaf ${\cal O}^{mult}$ of locally weakly
even-periodic commutative ring spectra on the \'etale site of
$\multmod$. In particular, given an \'etale map $\Spec(R) \to \Mell$
classifying a group scheme $\mb G$ over $R$ locally isomorphic to $\mb
G_m$, ${\cal O}^{mult}(R)$ is a weakly even-periodic spectrum
realizing $\widehat{\mb G}$.
\end{prop}

\begin{rmk}
The values of ${\cal O}^{mult}$ on affine schemes are forms of $K$-theory
\cite{morava-forms}: the analogues of elliptic cohomology theories
where forms of $\mb G_m$ replace elliptic curves.
\end{rmk}

We will require specific knowledge about how this presheaf is
constructed on affine opens of $\multmod$.

\begin{prop}
Let $Spec(S) \to \multmod$ be an \'etale map, corresponding to an
element $\alpha \in H^1_{et}(\Spec(S), C_2)$ that classifies a
$C_2$-Galois extension $T/S$. Then $S$ and $T$ are \'etale over $\mb
Z$, and there is a natural zigzag equivalence of commutative ring
spectra 
\[
{\cal O}^{mult}(S,\alpha) \simeq (K \wedge^{\mb L} (\mb S \otimes T))^{hC_2}.
\]
Here $\mb S \otimes T$ is the spectrum of Corollary~\ref{cor:etalespheres},
and $C_2$ has the diagonal action on $K \wedge^{\mb L} (\mb S \otimes T)$.
\end{prop}

\begin{rmk}
The group $C_2$ acts on the coefficient ring $K_* \otimes T$ with no
higher cohomology. Therefore, we have
\[
\pi_* {\cal O}^{mult}(S,\alpha) = (K_* \otimes T)^{C_2},
\]
and rationalization commutes with the homotopy fixed point
construction.
\end{rmk}

The global section object ${\cal O}^{mult}(\multmod)$ is the periodic real
$K$-theory spectrum $KO \simeq K^{hC_2}$.

We now must recall some of the main theorems about topological modular
forms.

\begin{thm}[{\cite{mark-construct}}]
There exists a fibrant presheaf of elliptic commutative ring spectra
${\cal O}^{et}$ on the small \'etale site of $\Mell$. In
particular, given an \'etale map $\Spec(R) \to \Mell$ classifying an
elliptic curve ${\cal E}$ over $R$, ${\cal O}^{et}(R)$ is an elliptic
spectrum realizing ${\cal E}$.
\end{thm}

\begin{defn}
Let $\Tmf = \Gamma(\Mell, {\cal O}^{et})$, and let $\tmf$ be
its connective cover. The values of ${\cal O}^{et}$ naturally take values
in commutative $\tmf$-algebras.
\end{defn}

\begin{defn}
Let ${\cal O}^{smooth}$ be the restriction of ${\cal O}^{et}$ to
a presheaf of commutative $\tmf$-algebras on the small \'etale site of
$\mell$, the open substack parametrizing smooth elliptic curves.
\end{defn}

\begin{prop}
\label{prop:k0tmf}
The rationalized functors ${\cal O}^{et}_{\mb Q}$ and ${\cal
  O}^{smooth}_{\mb Q}$ are formal.
\end{prop}

\begin{proof}
By Proposition~\ref{prop:rationaletale}, the functor ${\cal
  O}^{et}_{\mb Q}$ takes values in the category of commutative
$\tmf_{\mb Q}$-algebras such that map $\pi_* \tmf_{\mb Q} \to \pi_*
{\cal O}^{et}_{\mb Q}(R)$ is
\'etale. Proposition~\ref{prop:etaleformality} then implies the
desired result once we have shown that $\tmf_{\mb Q}$ is formal.

The ring $\pi_* \tmf_{\mb Q}$ is a polynomial algebra on the
classes $c_4$ and $c_6$. By the universal property of the free
algebra, up to equivalence we can construct a diagram
\[
H\pi_* \tmf \leftarrow \mb P(S^8 \vee S^{12}) \to
\tmf,
\]
and by Remark~\ref{rmk:rationalpolynomials} the rationalized diagram
\[
H\pi_* \tmf_{\mb Q} \leftarrow \mb P(S^8 \vee S^{12})_{\mb Q} \to
\tmf_{\mb Q}
\]
consists of equivalences of commutative ring spectra.
\end{proof}

The spectrum $\tmf$ is explicitly constructed so that the following
theorem holds.

\begin{prop}
\label{prop:k1tmf}
For any prime $p$, the $p$-adic $K$-theory of $\tmf$ is Katz's ring
$V$ of $p$-complete generalized modular functions \cite{katz-higher},
parametrizing isomorphism classes of elliptic curves ${\cal E}$ over
$p$-complete rings with a chosen isomorphism $\widehat{\mb G}_m \to
\widehat{{\cal E}}$.
\end{prop}

The following assembles applications of Goerss-Hopkins obstruction
theory from \cite[\S7]{mark-construct}.

\begin{prop}
\label{prop:obstructiontheory}
  Let $\Spf(R)$ be a formal scheme with ideal of definition 
  $(p)$, and let $\Spf(R) \to \comp{(\Mell)}_p$ be an \'etale formal
  affine open classifying an elliptic curve ${\cal E}/R$ with no
  supersingular fibers.  Then there is a lift of this data to a
  $K(1)$-local elliptic commutative $\tmf$-algebra $E$
  realizing ${\cal E}$.

  Given a second torsion-free $p$-complete ring $R'$ with a map
  $\Spf(R') \to \comp{(\Mell)}_p$ classifying an elliptic curve ${\cal
    E}'$ with no supersingular fibers, realized by an elliptic
  commutative $\tmf$-algebra $E'$, the space $\Map_{comm/\tmf}(E, E')$
  of $K(1)$-local commutative $\tmf$-algebra maps is homotopically
  discrete and equivalent to the set $\Hom_{\Mell}(\Spf(R'), \Spf(R))$
  of isomorphism classes of pullback diagrams
\[
\xymatrix{
{\cal E}' \ar[r] \ar[d] &
{\cal E} \ar[d] \\
\Spf(R') \ar[r] &
\Spf(R).
}
\]
\end{prop}

\begin{proof}
The necessary details for these results are proven within\cite[\S 7,
Step 2]{mark-construct} (just after Remark~7.15). There it is shown
that the \'etale condition implies that the higher Goerss-Hopkins
obstruction groups for $\tmf$-algebras all vanish in this setting. As
a result, for such an ${\cal E}/R$ there always exists an elliptic
commutative $\tmf$-algebra $E$ realizing ${\cal E}$. Moreover, the
space $\Map_{comm/\tmf}(E, E')$ is homotopically discrete: more
specifically, it is homotopy equivalent to a particular set of
algebraic maps that we will now describe.

Given any torsion-free $p$-complete ring $R$ with a map
$\Spf(R) \to \comp{(\Mell)}_p$ classifying an elliptic curve ${\cal
E}$ with no supersingular fibers, realized by an elliptic cohomology
theory $E$, there is a $p$-adically complete ring $W =
\pi_0 \comp{(K \wedge E)}_p$ which is the universal $R$-algebra where
${\cal E}$ comes equipped with a choice of isomorphism $\widehat{\mb
  G}_m \to \widehat{E}$. As $V$ is the universal $p$-complete ring
possessing such an elliptic curve, there is a natural map $\phi\co V
\to W$. The ring $W$ comes equipped with natural operators: an action
of $\mb Z_p^\times$ by Adams operations, and with a ring homomorphism
$\psi^p\co W \to W$ such that $\psi^p(x) = x^p + p \theta(x)$.

Given elliptic commutative $\tmf$-algebras as in the problem, the
Goerss-Hopkins obstruction theory shows that $p$-completed $K$-theory
gives a homotopy equivalence
\[
\Map_{comm/\tmf}(E,E') \simeq \Hom_{\psth\alg/V}(W,W')
\]
from the space of maps of $p$-complete $\tmf$-algebras $E \to E'$
to the discrete set of maps $W \to W'$ on $p$-adic $K$-theory which
are maps of $V$-algebras that respect the actions of $\mb Z_p^\times$
and $\theta$.

We will now explain how these imply the statement above.

The map $\Spf(W) \to \Spf(R)$ classifying isomorphisms $\widehat{\mb
  G}_m \to \widehat{E}$ is a principal torsor for $\Aut(\widehat{\mb
  G}_m)$, and hence a principal $\mb Z_p^\times$-torsor. If $R$ is
\'etale over $\Mell$, the map $V \to W$ is also \'etale. As a result,
for any $U$ which is $p$-adically complete there are always unique
lifts in diagrams of the following form:
\[
\xymatrix{
V \ar[rr]^{f} \ar[d]_\phi &&
U \ar@{>>}[d] \\
W \ar[rr]_-{\bar g} \ar@{.>}[urr]_{g}&&
U/(p)
}
\]
We will apply this twice.

First, by taking $f = \phi \psi^p$ and $\bar g = Frob$ we find that
there is a unique extension of $\psi^p$ (and $\theta$) to
$W$.

Second, if we start with one such diagram where $U$ has a lift
$\psi^p$ of Frobenius such that $f \psi^p = \psi^p f$, then we may
construct a new diagram by replacing $f$ with $\psi^p f$ and $\bar g$
with $Frob \cdot \bar g$. Then both $g \psi^p$ and $\psi^p g$ are
lifts making the diagram commute, so they are equal.

As a result, for such algebras $W$ and $W'$ a map of $V$-algebras
automatically commutes with $\theta$, and so the set
$\Hom_{\psth\alg/V}(W,W')$ is isomorphic to the set
$\Hom_{comm/V}(W,W')^{\mb Z_p^\times}$ of $\mb Z_p^\times$-equivariant
$V$-algebra maps.

The universal property of $V$ implies that the set of $V$-algebra maps
$W \to W'$ is the same as the set of ring maps $W \to W'$ together
with a chosen isomorphism ${\cal E} \otimes_R W' \to {\cal E}'
\otimes_{R'} W'$ respecting the isomorphisms of the formal parts with
$\widehat{\mb G}_m$. The universal properties of $W$ and $W'$,
parametrizing isomorphisms of the formal groups of ${\cal E}$ and
${\cal E'}$ with $\widehat{\mb G}_m$, together imply that a map is $\mb
Z_p^\times$-equivariant if and only if it descends to a map $R \to
R'$, together with an isomorphism ${\cal E} \otimes_R R' \to {\cal
  E}'$ over $R'$ expressing the desired pullback.
\end{proof}

\section{Construction of the presheaf ${\cal O}$}
\label{sec:construction}

\subsection{Construction at the cusp} 
\label{sec:cusps}

Corollary~\ref{cor:logetalecusp} will allow us to define
our presheaf ${\cal O}$ in a formal neighborhood of the cusp by
restricting our attention to a subcategory of log schemes over
$\mtate$ (\S\ref{sec:tate-moduli}).

\begin{defn}
The category ${\cal C}$ has, as objects, pairs $(m,S)$ of a
nonnegative integer $m$ and a connected, \'etale $\mb
Z[1/m]\pow{q^{1/m}}$-algebra $S$, complete and separated in the
topology generated by the ideal $(q)$.  We equip $S$ with the
logarithmic structure $\langle q^{1/m} \rangle$.  The maps in ${\cal
  C}$ are defined so that the forgetful functor to formal schemes over
$\mtate$, sending $(m,S)$ to $\Spec(S,\langle q^{1/m} \rangle)$, is
fully faithful.

\end{defn}

We will often leave $m$ implicit when describing objects of ${\cal
  C}$.  It can either be recovered as the ramification index of the ideal
$(q)$ or from the logarithmic structure.

\begin{defn}
There is a functor from ${\cal C}$ to affine \'etale opens of
$\multmod$, sending $(m,S)$ to the quotient $\overline S = S /
(q^{1/m})$ by the defining ideal.
\end{defn}

By Proposition~\ref{cor:logetalecusp}, the category ${\cal C}$
inherits a Grothendieck topology equivalent to the category of formal
schemes log-\'etale over $\mtate$.  To define the sheaves we are
interested in, it then suffices to define them on ${\cal C}$ by
Corollary~\ref{cor:cohomologysseq}.  We will abuse notation by
simply writing $S$ for the log ring $(S,\langle q^{1/m}\rangle)$
equipped with a form of the Tate curve and a lift of the discriminant.

\begin{defn}
We define functors from ${\cal C}$ to commutative monoids by
\[
\mu_m(S) = \{\zeta \in S^\times\, |\, \zeta^m = 1 \}
\]
and
\[
A(m,S) = \mu_m(S) \times \left(\tfrac{1}{m} \cdot \mb N\right).
\]
We write a generic element in $A(m,S)$ as $\zeta q^{k/m}$, where
$\zeta \in \mu_m(S)$ and $k \in \mb N$.
\end{defn}

\begin{rmk}
The natural map $\mu_m(S) \to \mu_m(\overline S)$ is an isomorphism.
\end{rmk}

\begin{rmk}
We note that a map $(S,\langle q^{1/m}\rangle) \to (S', \langle
q^{1/dm} \rangle)$ over $\mtate$ is equivalent to a map $\bar S \to
\bar S'$ over $\multmod$, together with an extension in the diagram
\[
\xymatrix{
\mu_m(S) \times q^{\mb N} \ar[r] \ar[d] & \mu_{dm}(S') \times q^{\mb N}\ar[d] \\
A(m,S) \ar@{.>}[r] & A(dm,S').
}
\]
This is because the natural map of monoids $A(m,S) \to S$ generates the
logarithmic structure on $S$, and we must have $(f(q^{1/m}))^m =
(q^{1/dm})^{dm}$. Once the image of $q^{1/m}$ is defined, the
extension to the rest of the \'etale $\mb Z[1/m]\pow{q^{1/m}}$-algebra
$S$ is forced.
\end{rmk}

\begin{prop}
\label{prop:reduced-tate-functor}
Up to equivalence, there is a natural map
\[
\mb S[\mu_m(S)] \to {\cal O}^{mult}(\overline S)
\]
of presheaves of commutative ring spectra on ${\cal C}$. On $\pi_0$,
this is a realization of the composite map $\mb Z[\mu_m(S)] \to
S \to \overline S$.
\end{prop}

\begin{proof}
As $m$ acts invertibly on $\pi_* {\cal O}^{mult}(\overline S)$, we
have an equivalence of derived spaces of commutative ring maps
\[
\Map_{comm}(\mb S[\mu_m(S)], {\cal O}^{mult}(\overline S))
\simeq \Map_{comm}(\mb S[\mu_m(S)][1/m], {\cal O}^{mult}(\overline S)).
\]
On homotopy groups, the map $\pi_* \mb S \to \pi_* \mb
S[\mu_m(S)][1/m]$ is always an \'etale map because $\mu_m(S)$ has
order invertible in $\mb Z[1/m]$. By
Proposition~\ref{prop:etalerings}, the derived space of maps
$\Sigma^\infty_+ A^\times(S) \to {\cal O}^{mult}(\overline S)$ is
homotopically discrete and equivalent to the space of maps
$\mu_m(S) \to \pi_0 ({\cal O}^{mult}(\overline S))^\times
\cong \overline S^\times$.

By Corollary~\ref{cor:discreterectification}, we can therefore replace
these presheaves $\mb S[\mu_m(S)]$ and ${\cal O}^{mult}(\overline S)$
with equivalent ones possessing a natural transformation as
desired.
\end{proof}

\begin{thm}
\label{thm:tate-functor}
There exists a presheaf ${\cal O}^{Tate}$ of elliptic commutative ring
spectra on ${\cal C}$, together with a natural homotopy pushout
diagram of commutative ring spectra
\[
\xymatrix{
\mb S[A(m,S)] \ar[r] \ar[d] & {\cal O}^{Tate}(m,S) \ar[d] \\
\mb S[\mu_m(S)] \ar[r] & {\cal O}^{mult}(\overline S).
}
\]
These satisfy the following properties:
\begin{itemize}
\item ${\cal O}^{Tate}(m,S)$ realizes the form of the Tate curve on $S$;
\item the map $A(m,S) \to \pi_0 {\cal O}^{Tate}(m,S) \cong S$ is
  the map inducing the logarithmic structure on $S$; and
\item the map $\mb S[A(m,S)] \to \mb S[\mu_m(S)]$ is induced by the
  projection $A(m,S) \to \mu_m(S)_+$ sending all non-units to the
  basepoint.
\end{itemize}
\end{thm}

\begin{proof}
For any $r > 0$, write $A(m,S)/q^r$ for the quotient based monoid
sending $\zeta q^{k/m}$ to the zero element for $k/m \geq r$.

We define our presheaf as a homotopy limit of (derived) smash
products:
\[
{\cal O}^{Tate}(S) = \holim_r \left[{\cal O}^{mult}(\overline S)
\smsh{\mb S[\mu_m(S)]}^{\mb L} \mb S[A(m,S)/q^r]\right]
\]
On homotopy groups, we have
\[
\pi_* \left[{\cal O}^{mult}(\overline S) \smsh{\mb S[\mu_m(S)]}^{\mb L} \mb
S[A(m,S)/q^r]\right] \cong \pi_* {\cal O}^{mult}(\overline S)[q^{1/m}]/(q^r),
\]
and taking limits gives an isomorphism
\[
\pi_* {\cal O}^{Tate}(S) \cong \pi_* {\cal O}^{mult}(\overline S)\pow{q^{1/m}}.
\]
Moreover, because ${\cal O}^{Tate}(S)$ is an ${\cal
  O}^{mult}(\overline S)$-algebra, the formal group of ${\cal
  O}^{Tate}(S)$ carries a canonical isomorphism to the extension of
the form of the multiplicative formal group over $\overline S$.  We
may then compose this with the isomorphism between this formal group
and the formal group of the form of the Tate curve over $S$, making
this into a presheaf of elliptic spectra realizing the Tate curve.

The multiplication-by-$q^{1/m}$ map $A(m,S) \to A(m,S)$ induces a cofiber
sequence of $\mb S[A(m,S)]$-modules
\[
\mb S[A(m,S)] \overto^{q^{1/m}} \mb S[A(m,S)] \to
\mb S[\mu_m(S)],
\]
where the right-hand map (but not the identification of the fiber) is
natural in $(m,S)$.  This remains a cofiber sequence after completing with
respect to the powers of $q$, and so the homotopy pushout diagram
follows by associativity of the derived smash product.
\end{proof}

By construction, the functor ${\cal O}^{Tate}$ naturally takes values in the
category of algebras over
\[
KO\pow{q} = \holim_r KO \smsh{} \{1,q,q^2,\cdots,q^{r-1}\}_+.
\]
We may then apply Theorem~\ref{thm:wittenhalf} from the appendix and
Corollary~\ref{cor:cohomologysseq} to conclude the following.
\begin{cor}
\label{cor:tate-tmf-algebra}
The presheaf ${\cal O}^{Tate}$, up to equivalence, extends to a
presheaf ${\cal O}^{Tate}$ of elliptic commutative $\tmf$-algebras on
the small \'etale site of $\mtate$.
\end{cor}

\subsection{Patching in the $p$-complete smooth portion}
\label{sec:smooth}
In this section we describe the $p$-completion of our desired
structure presheaf. We will abuse notation by referring to the composite
\[
(R,M) \mapsto {\cal O}^{Tate}(\comp{R}_\Delta,M)
\]
as a presheaf ${\cal O}^{Tate}$ of elliptic commutative
$\tmf$-algebras on $\Wlog$, the category of affine \'etale opens of
$\Mlog$ classifying Weierstrass curves (\ref{def:weier}).

\begin{prop}
For objects $\Spec(R,M) \to \Mlog$ of $\Wlog$, the natural map
\[
  \comp{(\Delta^{-1} R)}_p \to \comp{(\Delta^{-1}
    \comp{R}_\Delta)}_p
\]
over $\Mlog$ lifts, up to equivalence, to a map of presheaves of
$p$-complete elliptic commutative $\tmf$-algebras on $\Wlog$:
\[
  \comp{{\cal O}^{smooth} (\Delta^{-1} R)}_p \to \comp{({\Delta^{-1} \cal O}^{Tate}(R,M))}_p
\]
\end{prop}

\begin{proof}
Let $v_1 \in R$ be a lift of the Hasse invariant, which vanishes on
supersingular curves; it maps to a unit in $\comp{R}_{(\Delta,p)}$ because
nodal elliptic curves are never supersingular. Therefore, we have a
factorization
\[
\comp{(\Delta^{-1} R)}_p \to \comp{((v_1 \Delta)^{-1} R)}_p \to 
 \comp{(\Delta^{-1} \comp{R}_\Delta)}_p,
\]
and so this is equivalent to producing a natural map
\[
\comp{{\cal O}^{smooth}((v_1 \Delta)^{-1} R)}_p \to \comp{(\Delta^{-1}
  {\cal O}^{Tate}(R,M))}_p.
\]
The rings $\comp{((v_1\Delta)^{-1} R)}_p$ satisfy the criteria for
Proposition~\ref{prop:obstructiontheory}, so the spaces of algebra
maps $\comp{{\cal O}^{smooth} ((v_1\Delta)^{-1} R)}_p \to
\comp{({\Delta^{-1} \cal O}^{Tate}(R',M'))}_p$ are homotopically
discrete and equivalent to the set of maps $(R,M) \to (R',M')$ over
$\Mlog$. By Corollary~\ref{cor:discreterectification}, up to
equivalence we may construct a genuinely commutative diagram as
desired.
\end{proof}

\begin{defn}
\label{def:pcomplete-sheaf}
  The presheaf $\comp{{\cal O}}_{p}$, of $p$-complete
  elliptic commutative $\tmf$-algebras on $\Wlog$, sends $\Spec(R,M)
  \to \Mlog$ to the homotopy pullback in the following diagram:
\[
\xymatrix{
\comp{{\cal O}}_p(R,M) \ar@{.>}[d] \ar@{.>}[r] &
\comp{{\cal O}^{Tate}(R,M)}_p \ar[d] \\
\comp{{\cal O}^{smooth}(\Delta^{-1} R)}_p \ar[r] & 
\comp{({\Delta^{-1} \cal O}^{Tate}(R,M))}_p
}
\]
\end{defn}

We have the following.
\begin{prop}
For any affine log scheme $(R,M) \in \Wlog$ carrying the generalized
elliptic curve ${\cal E}$, the spectrum $\comp{{\cal O}}_p(R)$ is an
elliptic cohomology theory associated to the elliptic curve ${\cal E}$
over $\comp{R}_p$.
\end{prop}

\begin{proof}
We consider the following diagram:
\begin{equation}
  \label{eq:bigpullback}
\xymatrix{
\comp{(\omega^{\otimes n})}_p \ar@{.>}[d] \ar@{.>}[r] &
\comp{(\omega^{\otimes n})}_{(p,\Delta)} \ar[d]
\\
\comp{(\Delta^{-1} (\omega^{\otimes n}))}_p \ar[r] &
\comp{(\Delta^{-1} \comp{(\omega^{\otimes n})}_\Delta)}_p
}
\end{equation}
Before $p$-adic completion, this square is bicartesian by
Proposition~\ref{prop:artinrees}. As the terms consist of torsion-free
groups, this property is preserved by $p$-adic completion.

Taking homotopy groups in Definition~\ref{def:pcomplete-sheaf}, the
Mayer-Vietoris sequence for pullbacks shows that the odd homotopy
groups of the pullback are zero and that the even homotopy groups are
the tensor powers of $\omega$, as desired. By
Lemma~\ref{lem:orientationpullback}, the result is an elliptic
commutative ring spectrum.
\end{proof}

\subsection{Patching in the rational part}
\label{sec:rational}  

In this section, we will use an arithmetic square to construct ${\cal
  O}_{\mb Q}$ and ${\cal O}$.

For an object $(R,M) \in \Wlog$, we have a well-defined value of ${\cal
  O}_{\mb Q}$ on the restriction $\Delta^{-1} R$ to the smooth locus.

\begin{prop}
Up to equivalence, there exists a map
\[
{\cal O}^{smooth}(\Delta^{-1} R)_{\mb Q} \to \Delta^{-1} {\cal
  O}^{Tate}(R,M)_{\mb Q}
\]
of presheaves of elliptic commutative $\tmf$-algebras on
$\Wlog$ that fits into the following commutative diagram:
\[
\xymatrix{
{\cal O}^{smooth}(\Delta^{-1} R)_{\mb Q} \ar@{.>}[r] \ar[d] &
\Delta^{-1} {\cal O}^{Tate}(R,M)_{\mb Q} \ar[d] \\
(\prod_p \comp{{\cal O}^{smooth}(\Delta^{-1} R)}_p)_{\mb Q} \ar[r] &
\Delta^{-1} (\prod_p \comp{{\cal O}^{Tate}(R,M)}_p)_{\mb Q}.
}
\]
\end{prop}

\begin{proof}
By Proposition~\ref{prop:rationaletale}, the functor ${\cal
  O}^{smooth}_{\mb Q}$ takes values in the category of commutative
$\tmf_{\mb Q}$-algebras such that map $\pi_* \tmf_{\mb Q} \to \pi_*
{\cal O}^{smooth}_{\mb Q}(R)$ is
\'etale. Proposition~\ref{prop:etalerings} then implies that the space
of maps of elliptic commutative $\tmf$-algebras out of ${\cal
O}^{smooth}(\Delta^{-1} R)_{\mb Q}$ is always homotopically discrete 
and equivalent to the set of algebraic maps. By
Corollary~\ref{cor:discreterectification}, we may construct the
desired lift up to equivalence.
\end{proof}

\begin{defn}
The presheaf ${\cal O}_{\mb Q}$ of rational elliptic commutative
$\tmf$-algebras on $\Wlog$ sends $\Spec(R,M) \to \Mlog$ to the homotopy
pullback in the diagram
\[
\xymatrix{
{\cal O}_{\mb Q}(R,M) \ar[r] \ar[d] &
{\cal O}^{Tate}(R,M)_{\mb Q} \ar[d] \\
{\cal O}^{smooth}(\Delta^{-1} R)_{\mb Q} \ar[r] &
\Delta^{-1}{\cal O}^{Tate}(R,M)_{\mb Q}.
}
\]
\end{defn}

\begin{prop}
For any affine log scheme $(R,M) \in \Wlog$ carrying the generalized
elliptic curve ${\cal E}$, the spectrum ${\cal O}_{\mb Q}(R,M)$ is an
elliptic spectrum associated to the elliptic curve ${\cal E}$
over $R_{\mb Q}$, and maps in $\Wlog$ become maps of elliptic spectra.

Up to equivalence, there is a map
\[
{\cal O}_{\mb Q}(R,M) \to \left(\prod_p \comp{{\cal O}}_p(R,M)\right)_{\mb Q}
\]
of presheaves of elliptic commutative $\tmf$-algebras.
\end{prop}

\begin{proof}
For any $n$, we have a bicartesian square
\[
\xymatrix{
\omega^{\otimes n}_{\mb Q} \ar[r] \ar[d] &
(\comp{(\omega^{\otimes n})}_{\Delta})_{\mb Q} \ar[d]\\
\Delta^{-1} \omega^{\otimes n}_{\mb Q} \ar[r] &
\Delta^{-1}(\comp{(\omega^{\otimes n})}_{\Delta})_{\mb Q}\\
}
\]
which is obtained by rationalizing the corresponding bicartesian
square for the finitely generated module $\omega^{\otimes
  n}$.  Therefore, the homotopy groups of ${\cal O}_{\mb Q}(R,M)$ are
the modules $\omega_{\mb Q}^{\otimes n}$, and by
Lemma~\ref{lem:orientationpullback} this is an elliptic spectrum
realizing ${\cal E}$. 

To construct the given map, we note that we now have a natural diagram
of commutative $\tmf$-algebras:
\[
\xymatrix{
{\cal O}^{Tate}(R,M)_{\mb Q} \ar[r] \ar[d] &
\Delta^{-1}{\cal O}^{Tate}(R,M)_{\mb Q} \ar[d] &
{\cal O}_{\mb Q}(\Delta^{-1} R) \ar[l] \ar[d] \\
(\prod_p \comp{{\cal O}^{Tate}(R,M)}_p)_{\mb Q} \ar[r] &
\Delta^{-1} (\prod_p \comp{{\cal O}^{Tate}(R,M)}_p)_{\mb Q} &
(\prod_p \comp{{\cal O}}_p(\Delta^{-1}R))_{\mb Q} \ar[l]
}
\]
Taking homotopy pullbacks in rows gives the desired natural map.
\end{proof}

\begin{defn}
The presheaf ${\cal O}$ of elliptic commutative $\tmf$-algebras on
$\Wlog$ sends $\Spec(R,M)$ to the homotopy pullback in the diagram
\[
\xymatrix{
{\cal O}(R,M) \ar[r] \ar[d] &
\prod_p\comp{{\cal O}}_p(R,M) \ar[d] \\
{\cal O}_{\mb Q}(R,M) \ar[r] &
(\prod_p\comp{{\cal O}}_p(R,M))_{\mb Q}.
}
\]
\end{defn}

As a consequence of Lemma~\ref{lem:orientationpullback}, we find that
${\cal O}(R,M)$ is a functorial elliptic spectrum realizing the
elliptic curve ${\cal E}$ on $R$.

The natural map ${\cal O}^{Tate} \to {\cal O}^{mult}$ of
Proposition~\ref{prop:reduced-tate-functor} now allows us to evaluate
at the cusps.
\begin{prop}
\label{prop:eval-at-cusp}
There is a map
\[
{\cal O}(R,M) \to {\cal O}^{mult}(\overline R)
\]
of presheaves on $\Wlog$, where the range is the form of $K$-theory
associated to the cusp subscheme of $\Spec(R)$.
\end{prop}

By Proposition~\ref{prop:sheafifyme}, we may automatically extend
${\cal O}$ and ${\cal O}^{mult}$, as well as the map between them, to
presheaves of elliptic commutative $\tmf$-algebras on the log-\'etale
site of $\Mlog$ that satisfy homotopy descent.

\begin{thm}
\label{thm:realmain}
There exists a realization of the universal elliptic curve on the
small log-\'etale site of $\Mlog$ by a presheaf ${\cal O}$ of elliptic
commutative $\tmf$-algebras, satisfying homotopy descent for
hypercovers, together with a map ${\cal O} \to {\cal   O}^{mult}$
realizing evaluation at the cusp.
\end{thm}

The descent property now allows us to take sections on any
Deligne-Mumford stack equipped with a logarithmic structure which is
log-\'etale over $\Mlog$.

\section{$\Tmf$ with level structure}
\label{sec:results}

Having constructed our derived structure presheaf ${\cal O}$ in the
previous section, we can now evaluate it on the modular curves
$\MMM(\Gamma)$ from \S\ref{sec:modular}.

\begin{thm}
\label{thm:main}
There exists a contravariant functor $\Tmf$ from the category ${\cal
L}$ of Definition~\ref{def:levelstruct} to elliptic commutative
$\tmf$-algebras, taking a pair $(N,\Gamma)$ to an object
$\Tmf(\Gamma)$ which is $\mb Z[1/N]$-local.  When $N=1$ (and hence
$\Gamma$ is trivial) this recovers the nonconnective, nonperiodic
spectrum $\Tmf$.

For any such $\Gamma$, there is a spectral sequence
\[
H^s(\MMM(\Gamma); \omega^{\otimes t/2}) \Rightarrow \pi_{t-s} \Tmf(\Gamma),
\]
and for any $K \lhd \Gamma < \GL_2(\mb Z/N)$ the natural map
\[
\Tmf(\Gamma) \to \Tmf(K)^{h\Gamma/K}
\]
is an equivalence.

If $p\co \GL_2(\mb Z/NM) \to \GL_2(\mb Z/N)$ is the natural
projection, the map $\Tmf(\Gamma) \to \Tmf({p^{-1} \Gamma})$ is a
localization formed by inverting $M$.
\end{thm}

\begin{proof}
The functoriality of the modular curves $\MMM(\Gamma)$ as log-\'etale
objects over $\Mlog$ was discussed in
Proposition~\ref{prop:modulartower}.  Therefore, we may apply ${\cal
  O}$ to obtain a functorial diagram of commutative $\tmf$-algebras.  By
definition, the value on the terminal object is $\Tmf$.  The
statement about localizations is true due to the global section
functor commuting with homotopy colimits, since the map $M(\Gamma) \to
\mfg$ is tame \cite[4.12]{mathew-meier-affine}; the ring of
sections over a localization is the localization of the ring of
sections.

The spectral sequence for the cohomology of $\Tmf(\Gamma)$ is
Corollary~\ref{cor:cohomologysseq}, while the equivalence from
$\MMM(K)$ to the homotopy fixed-point object
$\MMM(\Gamma)^{h\Gamma/K}$ is Proposition~\ref{prop:galoissseq}.
\end{proof}

We can also evaluate at the cusps.

\begin{thm}
\label{thm:thmcusp}
Let $K(\Gamma)$ be the natural form of $K$-theory associated to the
cusp substack of $\MMM(\Gamma)$.  There is a natural transformation of
commutative $\tmf$-algebras
\[
\Tmf(\Gamma) \to K(\Gamma).
\]
\end{thm}

In particular, we can apply Theorem~\ref{thm:main} to the specific
cover $\MMM_1(3) \to \MMM_0(3)$, obtaining in particular an (almost)
integral lift of the work in \cite{tmforientation}.

\begin{thm}
\label{thm:tmfexists}
There exists a commutative $\tmf$-algebra $\tmf_0(3)$ (with $3$
inverted) whose homotopy groups form the ``positive'' portion of the
homotopy groups of $\TMF_0(3)$ described in \cite[\S
7]{mahowald-rezk-level3}.  This fits into a commutative diagram of
commutative $\tmf$-algebras
\[
\xymatrix{
\tmf_0(3) \ar[r] \ar[d] & ko[1/3] \ar[d] \\
\tmf_1(3) \ar[r] & ku[1/3].
}
\]
\end{thm}

\begin{proof}
The cohomology of the moduli stack $\MMM_1(3)$ was essentially
determined in \cite{mahowald-rezk-level3}; see also
\cite{tmforientation}.  The cohomology groups $H^0(\MMM_1(3);
\omega^{\otimes t})$ form the graded ring $\mb Z[1/3,a_1,a_3]$,
which come from the universal cubic curve
\[
y^2 + a_1 xy + a_3 y = x^3
\]
carrying a triple intersection with the line $y=0$.  The other
cohomology groups $H^s(\MMM_1(3); \omega^{\otimes t})$ are
concentrated in $s = 1$, $t \leq -4$, and so the positive-degree
homotopy groups of $\Tmf_1(3)$ form the graded ring $\mb
Z[1/3,a_1,a_3]$.  Moreover, the open subscheme ${\cal M}_1(3) \subset
\MMM_1(3)$ induces a map of commutative $\tmf$-algebras $\Tmf_1(3) \to
\TMF_1(3)$ which, on homotopy groups, inverts the elliptic
discriminant $\Delta = a_3^3(a_1^3 - 27a_3)$.

The homotopy fixed-point spectral sequence for the homotopy groups of
$\Tmf_0(3)$, in positive degrees, consists of the part of the
computation carried out by Mahowald-Rezk which involves no negative
powers of $\Delta$.  The portion of the spectral sequence with $s >
t-s \geq 0$ consists of spurious classes which are annihilated by
differentials, and so the portion of the spectral sequence with $s
\leq t-s$ converges to the homotopy groups of the connective cover
$\tmf_0(3)$, which is the ``positive'' portion of the computation
described in \cite[\S7]{mahowald-rezk-level3}.

Evaluating at the cusps, we obtain a diagram of commutative
$\tmf$-algebras
\[
\xymatrix{
\Tmf_0(3) \ar[r] \ar[d] & KO[1/3] \ar[d] \\
\Tmf_1(3) \ar[r] & KU[1/3] \times KU^\tau[1/3],
}
\]
a global version of the one described in \cite{tmforientation}, with
$KU^\tau$ a form of $K$-theory. Projecting away from the factor
$KU^\tau[1/3]$ and taking connective covers, we obtain a diagram of
commutative $\tmf$-algebras giving the desired connective lifts.
\end{proof}

\begin{rmk}
For sufficiently small subgroups $\Gamma < \GL_2(\mb Z/N)$, the
modular curve $\MMM(\Gamma)$ is genuinely a scheme, and the
cohomology of $\MMM(\Gamma)$ is concentrated in degrees $0$ and $1$.
The spectral sequence degenerates to isomorphisms, where the
cohomology in the following is implicitly cohomology of
$\MMM(\Gamma)$:
\begin{align*}
\pi_{2t} \Tmf(\Gamma) &\cong H^0(\omega^{\otimes t})\\
\pi_{2t-1} \Tmf(\Gamma) &\cong H^1(\omega^{\otimes t})
\end{align*}
The even-degree homotopy groups form the ring of modular forms for
$\Gamma$ over $\mb Z[1/N]$, and are concentrated in nonnegative
degrees.  Duality for $H^1$ (specifically, Grothendieck-Serre duality
\cite{hartshorne-residuesduality}) takes the form of an
exact sequence
\[
0 \to \Ext(H^1(\kappa \otimes \omega^{\otimes(-t)}),\mb Z[1/N]) \to
H^1(\omega^{\otimes t}) \to \Hom(H^0(\kappa \otimes \omega^{\otimes
  (-t)}), \mb Z[1/N]) \to 0.
\]
In these cases we have an isomorphism of $\omega^{\otimes 2}$ with the
logarithmic cotangent complex, which is the twist $\kappa(D)$ of the
canonical bundle by the cusp divisor. This allows us to recast $H^1$
as coming from an exact sequence
\[
0 \to \Ext(H^1(\omega^{\otimes(2-t)}(-D)),\mb Z[1/N]) \to
H^1(\omega^{\otimes t}) \to \Hom(H^0(\omega^{\otimes
  (2-t)}(-D)), \mb Z[1/N]) \to 0.
\]

Any $p$-torsion elements in $H^1$ arise due to forms of weight $k$
mod-$p$ that do not lift to integral forms, and outside weight $1$
this does not occur as a consequence of the Riemann-Roch formula.

Degree considerations imply that the only possible nonzero homotopy
group in positive, odd degree is $\pi_1(\Tmf(\Gamma))$.  This has
possible contributions from both the dual of the space
$H^0(\omega(-D))$, parametrizing cuspforms of weight $1$ and level
$\Gamma$, and the Pontrjagin dual of its torsion.

Using a Postnikov tower, we can eliminate some of this torsion from
the homotopy groups of $\Tmf(\Gamma)$.  However, it is not clear if
there is a conceptually correct way to do so.  The torsion of
$H^1(\omega)$ and $H^1(\omega(-D))$ are Pontrjagin dual, and measure
the failure of weight-$1$ forms with level $\Gamma$ to lift.  The
story of these non-liftable forms of weight one seems to be just
beginning to emerge~\cite{buzzard-weightone,schaeffer-thesis}.
\end{rmk}

\appendix
\section{Appendix: The Witten genus}
\label{sec:witten}

The goal of this section is to construct a map of commutative ring
spectra
\[
\tmf \to KO\pow{q}
\]
which, on homotopy groups, factors the Witten genus $MSpin_* \to \mb
Z\pow{q}$.  Here the power series notation $KO\pow{q}$ is shorthand
for the homotopy limit of the monoid algebras
\[
\holim_r KO \smsh{} \{1, q, \cdots, q^{r-1}\}_+
\]
where $q^{r}$ is identified with the basepoint (as in \S
\ref{sec:cusps}).  The main result (Theorem~\ref{thm:wittenhalf}) is
well-known and featured prominently in earlier, unpublished,
constructions of $\tmf$, but to the knowledge of the authors it does
not appear in the literature.  The relation of the Tate curve to power
operations has been extensively explored, especially in this context
by Baker \cite{baker-hecke}, Ando \cite{ando-ellipticpowerops},
Ando-Hopkins-Strickland \cite{ando-hopkins-strickland-witten}, and
Ganter \cite{ganter-orbifoldtate}.

\begin{defn}
For a chosen prime $p$, the $p$-adic $K$-theory of a spectrum $X$ is
\[
K^\vee_*(X) = \pi_* L_{K(1)}(K \wedge X) = \pi_* \holim_k (K/p^k \wedge
X).
\]
In particular, the coefficient ring $K^\vee_*$ is the graded ring $\mb
Z_p[\beta^{\pm 1}]$.
\end{defn}

Here $K/p^k$ is the mapping cone of the multiplication-by-$p^k$
endomorphism of $K$, having a long exact sequence
\begin{equation}
\label{eq:longexact}
\cdots \to K^\vee_*(X) \overto^{p^k}
K^\vee_*(X) \to
\pi_* (K/p^k \smsh{} X) \to \cdots
\end{equation}
which is natural in $X$.

\begin{rmk}
As $K$-modules and $KO$-modules are automatically $E(1)$-local,
$K(1)$-localizations and $p$-completions are equivalent on them.
\end{rmk}

We recall the following result, which was classically used as a
definition of $K(1)$-local $\tmf$ at the prime $2$.

\begin{prop}[{\cite{hopkins-k1-local,laures-k1-local-tmf}}]
\label{prop:tmfpushout}
At $p=2$, there are homotopy pushout diagrams
\[
\xymatrix{
L_{K(1)}\mb P(S^{-1}) \ar[r]^-0 \ar[d]_\zeta &
L_{K(1)}\mb S \ar[d] &
L_{K(1)}\mb P(S^0) \ar[r]^-0 \ar[d]_{\theta(f) - h(f)} &
L_{K(1)}\mb S \ar[d] \\
L_{K(1)}\mb S \ar[r] &
T_\zeta &
T_\zeta \ar[r] &
L_{K(1)}\tmf
}
\]
in the category of $K(1)$-local commutative ring spectra. Here $\zeta$
is a topological generator of $\pi_{-1} L_{K(1)}\mb S \cong \mb Z_2$;
$f$ is an element in $\pi_0 T_\zeta$; and $h(x)$ is a $2$-adically
convergent power series such that for any $K(1)$-local elliptic
commutative ring spectrum $E$, any map of commutative ring spectra
$T_\zeta \to E$ automatically sends $\theta(f)$ and $h(f)$ to the same
element.
\end{prop}

We first need to identify the $p$-adic $K$-theory of $KO\pow{q}$.

\begin{prop}
\label{prop:tatektheory}
For any prime $p$, the $p$-adic $K$-theory of $KO\pow{q}$ is the ring
\[
K^\vee_* (KO\pow{q}) \cong 
\Map_c(\mb Z_p^\times, K^\vee_*\pow{q})^{\{\pm 1\}}.
\]
Here the group $\{\pm 1\}\subset \mb Z_p^\times$ acts by conjugation
on the group of continuous homomorphisms, and the ring
$K^\vee_*\pow{q}$ is given the $p$-adic topology.

This is the universal $p$-complete $\mb Z\pow{q}$-algebra with an
isomorphism class of pairs of an invariant $1$-form on the Tate curve
$T$ and an identification $\widehat{\mb G}_m \stackrel{\sim}{\overto}
\widehat{T}$ between the formal multiplicative group and the formal
group of the Tate curve. The map $V \to K^\vee_0(KO\pow{q})$
determined by this is a map of $\psth$-algebras.
\end{prop}

\begin{proof}
We recall from \cite{hopkins-k1-local} or
\cite[9.2]{ando-hopkins-rezk-kotheory} that the map of
$\psth$-algebras 
\[
K^\vee_* KO \to K^\vee_* K
\]
is the inclusion
\[
\Map_c(\mb Z_p^\times, K^\vee_*)^{\{\pm 1\}} \hookrightarrow
\Map_c(\mb Z_p^\times, K^\vee_*)
\]
of sets of continuous maps. The long exact sequence
(\ref{eq:longexact}) gives an identification
\begin{align*}
\pi_* (K/p^k \wedge KO)
&\cong \Map_c(\mb Z_p^\times, (K_*)/p^k)^{\{\pm 1\}}\\
&= \colim_{m} \Map((\mb Z/p^m)^\times, (K_*)/p^k)^{\{\pm 1\}}.
\end{align*}

The graded $\pi_* KO$-module $\pi_* (KO\pow{q}) \cong \pi_* KO
\otimes_{\mb Z} \mb Z\pow{q}$ is flat, and so the isomorphism
\[
K/p^k \wedge KO\pow{q} \cong (K/p^k \wedge KO) \smsh{KO} KO\pow{q}
\]
can be re-expressed as an isomorphism
\[
\pi_*(K/p^k \wedge KO\pow{q}) \cong \Map_c(\mb Z_p^\times,
K_*\pow{q}/p^k)^{\{\pm 1\}}
\]
by the K\"unneth formula. Taking limits gives the desired formula for
the $p$-adic $K$-theory.

The action of the group $\mb Z_p^\times$ on $\Map_c(\mb Z_p^\times,
K^\vee_*)^{\{\pm 1\}}$, by premultiplication on the source, is
compatible with this isomorphism, and therefore determines the action
of $\mb Z_p^\times$ on $K^\vee_*(KO\pow{q})$: it is coinduced from the
action of the subgroup $\{\pm 1\}$ on $K^\vee_*\pow{q}$. The ring
$\Map_c(\mb Z_p^\times, K^\vee_*\pow{q})$ is the universal ring
classifying isomorphisms $\widehat {\mb G}_m \to \widehat T$ together
with a choice of invariant $1$-form, as any such isomorphism differs
from the canonical one by a locally constant function to $\mb
Z_p^\times$. As a $\{\pm 1\}$-equivariant algebra over the ring of
invariants $K^\vee_* KO\pow{q}$ it is isomorphic to $K^\vee_*
KO\pow{q} \times K^\vee_* KO\pow{q}$, and so the ring of invariants
classifies the quotient by $Aut(T) \cong \{\pm 1\}$. There is a map $V
\to K^\vee_0 KO\pow{q}$ determined by the universal property of $V$.

The element $q \in K^\vee_0 KO\pow{q}$, since it lifts to an element
in $\pi_0 KO\pow{q}$, is acted on trivially by $\mb
Z_p^\times$. Moreover, the extended power operation $\psi^p$ lifts to
the corresponding power operation on the discrete monoid $\mb N$,
which sends $q$ to $q^p$. The resulting map on $\mb Z\pow{q}$
classifies the quotient of the Tate curve by the canonical subgroup
$\mu_p$ of its formal group (\S\ref{sec:tate-curve}), and thus the map
$V \to K^\vee_0(KO\pow{q})$ preserves the operation $\psi^p$ (and
hence $\theta$).
\end{proof}

\begin{prop}
\label{prop:evencusp}
At the prime $2$, there exists a map of $K(1)$-local commutative ring
spectra
\[
L_{K(1)} \tmf \to L_{K(1)} KO\pow{q}
\]
which, on $2$-adic $K$-homology, induces the map
\[
V \to K^\vee_0(KO\pow{q})
\]
from Proposition~\ref{prop:tatektheory}.
\end{prop}

\begin{proof}
We will use the description of $K(1)$-local $\tmf$ from
Proposition~\ref{prop:tmfpushout} to construct this map.

As $\comp{KO\pow{q}}_2$ is the $K(1)$-localization of $KO\pow{q}$ and
has trivial $\pi_{-1}$, we have a map of commutative ring spectra
$T_\zeta \to \comp{KO\pow{q}}_2$.  The composite map $T_\zeta \to
\comp{K\pow{q}}_2$ detects the effect on $\pi_0$, and is a map to an
elliptic cohomology theory, where the latter carries the Tate curve
over the power series ring $\comp{\mb Z\pow{q}}_2$.  Therefore, the
element $\theta(f) - h(f)$ automatically maps to zero, and we obtain
an extension $L_{K(1)}\tmf \to L_{K(1)}KO\pow{q}$.
\end{proof}

\begin{prop}
\label{prop:oddcusp}
At any odd prime $p$, there exists a map of $K(1)$-local commutative
ring spectra
\[
L_{K(1)} \tmf \to L_{K(1)} KO\pow{q}
\]
which, on $p$-adic $K$-theory, induces the map
\[
V \to K^\vee_0(KO\pow{q})
\]
from Proposition~\ref{prop:tatektheory}.
\end{prop}

\begin{proof}
As $KO\pow{q}$ is the homotopy fixed-point spectrum of the action of
$\{\pm 1\}$ on $K\pow{q}$, we have an equivalence
\[
\Map_{comm}(L_{K(1)} \tmf, L_{K(1)} KO\pow{q}) \simeq
\Map_{comm}(L_{K(1)} \tmf,
L_{K(1)} K\pow{q})^{h\{\pm 1\}}.
\]

The Goerss-Hopkins obstruction theory computing this space of maps of
$K(1)$-local commutative ring spectra produces obstructions in
Andr\'e-Quillen cohomology groups. There is a fringed spectral
sequence with $E_2$-term given by
\[
E_2^{s,t} =
\begin{cases} 
\Hom_{\psth\alg/K^\vee_*} (K^\vee_* \tmf, K^\vee_* K\pow{q})&\text{if }(s,t)=(0,0),\\
H^s_{\psth\alg/K^\vee_*} (K^\vee_* \tmf, \Omega^t K^\vee_* K\pow{q})&\text{otherwise}.
\end{cases}
\]
This spectral sequence converges to $\pi_{t-s} \Map_{comm}(L_{K(1)}
\tmf, L_{K(1)} K\pow{q})$. By \cite[7.5]{mark-construct}, the fact
that $V$ is formally smooth over $\mb Z_p$ implies that the
obstruction groups
\[
H^s_{\psth\alg/K^\vee_*} (K^\vee_* \tmf, \Omega^t
K^\vee_* K\pow{q})
\]
are trivial for $s > 1$ or $t = 0$.

In particular, the homotopy groups $\pi_{t} \Map_{comm}(L_{K(1)} \tmf,
L_{K(1)} K\pow{q})$ are $p$-adically complete abelian groups for any
choice of basepoint, and so the homotopy fixed-point spectral sequence
for the action of the group $\{\pm 1\}$ degenerates. We find that the
set of path components is
\[
\pi_0 \Map_{comm}(L_{K(1)} \tmf,
L_{K(1)} K\pow{q}) \cong \Hom_{\psth\alg/K^\vee_*}(V, K_*\pow{q})^{\{\pm 1\}},
\]
and so the map of Proposition~\ref{prop:tatektheory} has a lift which
is unique up to homotopy.
\end{proof}

\begin{prop}
\label{prop:rationalcusp}
There exists a map of rational commutative ring spectra 
\[
\tmf_{\mb Q} \to (KO\pow{q})_{\mb Q}
\]
which, on homotopy groups, is given by a map
\[
\mb Q[c_4, c_6] \to \mb Q \otimes \mb Z\pow{q}[\beta^{\pm 2}]
\]
sending $c_4$ and $c_6$ to their $q$-expansions.  The two maps
\[
\tmf \to (\prod_p \comp{KO\pow{q}}_p)_{\mb Q},
\]
induced by this map and the maps constructed in
Propositions~\ref{prop:evencusp} and \ref{prop:oddcusp}, are homotopic
as maps of commutative ring spectra.
\end{prop}

\begin{proof}
The elements $c_4$ and $c_6$ can be realized as maps $S^8 \to
\tmf_{\mb Q}$ and $S^{12} \to \tmf_{\mb Q}$ respectively.  The induced
map of commutative ring spectra $\mb P_{\mb Q}(S^8 \vee S^{12}) \to
\tmf_{\mb Q}$ is a weak equivalence, and so homotopy classes of
commutative ring spectrum maps $\tmf_{\mb Q} \to (KO\pow{q})_{\mb Q}$
are defined uniquely, up to homotopy, by specifying the images of
$c_4$ and $c_6$.

Homotopy classes of maps of commutative ring spectra $\tmf \to
(\prod_p \comp{KO\pow{q}}_p)_{\mb Q}$ are the same as maps $\tmf_{\mb
  Q} \to (\prod_p \comp{KO\pow{q}}_p)_{\mb Q}$, and are similarly
determined by the images of $c_4$ and $c_6$.  Therefore, as the
$K(1)$-local and rational constructions are both obtained by
$q$-expansion in a neighborhood of the Tate curve, the resulting pair
of maps are homotopic as maps of commutative ring spectra.
\end{proof}

\begin{thm}
\label{thm:wittenhalf}
There exists a map of commutative ring spectra
\[
\tmf \to KO\pow{q}
\]
compatible with the $K(1)$-local and rational maps 
constructed in Propositions~\ref{prop:evencusp}, \ref{prop:oddcusp},
and \ref{prop:rationalcusp}.
\end{thm}

\begin{proof}
We can express the spectrum $KO\pow{q}$ as a homotopy pullback in the
following arithmetic square of commutative ring spectra:
\[
\xymatrix{
KO\pow{q} \ar[r] \ar[d] &
\prod_p \comp{KO\pow{q}}_p \ar[d] \\
(KO\pow{q})_{\mb Q} \ar[r] &
(\prod_p \comp{KO\pow{q}}_p)_{\mb Q}
}
\]
However, from Propositions~\ref{prop:evencusp}, \ref{prop:oddcusp},
and \ref{prop:rationalcusp} we obtain maps from $\tmf$ to the rational
and $p$-completed entries which are homotopic, and therefore a map
from $\tmf$ to the homotopy pullback.
\end{proof}

\begin{rmk}
As the spectrum $(\prod_p \comp{KO\pow{q}}_p)_{\mb Q}$ has trivial
homotopy groups in degrees $9$ and $13$, the path components of the
mapping space
\[
\Map_{}\left(\tmf,(\prod_p \comp{KO\pow{q}}_p)_{\mb Q}\right)
\]
are all simply connected.  The Mayer-Vietoris square of mapping spaces
shows that there is a unique homotopy class of map of commutative ring
spectra from $\tmf$ to the pullback.
\end{rmk}

\bibliography{masterbib}

\end{document}